\documentclass[12pt,american,english]{article}
\usepackage{mathpazo}
\usepackage[T1]{fontenc}
\usepackage[latin9]{inputenc}
\usepackage{babel}
\usepackage{enumitem}
\usepackage{amsmath}
\usepackage{amsthm}
\usepackage{amssymb}
\usepackage{stmaryrd}
\usepackage{stackrel}
\usepackage{setspace}
\usepackage{wasysym}
\usepackage[authoryear]{natbib}

\usepackage{bussproofs}
\usepackage{turnstile}
\usepackage{marvosym}
\usepackage{adjustbox}
\usepackage{titletoc}

\usepackage[unicode=true,pdfusetitle,
 bookmarks=true,bookmarksnumbered=true,bookmarksopen=true,bookmarksopenlevel=2,
 breaklinks=false,pdfborder={0 0 0},pdfborderstyle={},backref=false,colorlinks=false]
 {hyperref}

\makeatletter
\numberwithin{equation}{section}
\numberwithin{figure}{section}
\newlength{\lyxlabelwidth}      
	\newenvironment{elabeling}[2][]%
	{\settowidth{\lyxlabelwidth}{#2}
		\begin{description}[font=\normalfont,style=sameline,
			leftmargin=\lyxlabelwidth,#1]}
	{\end{description}}
\theoremstyle{plain}
\newtheorem{thm}{\protect\theoremname}
\theoremstyle{definition}
\newtheorem{defn}[thm]{\protect\definitionname}
\theoremstyle{remark}
\newtheorem{notation}[thm]{\protect\notationname}
\theoremstyle{plain}
\newtheorem{prop}[thm]{\protect\propositionname}
\theoremstyle{plain}
\newtheorem{lem}[thm]{\protect\lemmaname}

\hypersetup{pdfstartview={FitH}}

\usepackage[all]{nowidow}
\renewcommand{\makehor}[4]
    {\ifthenelse{\equal{#1}{n}}{\!\hstretch{2.3}{\raisebox{-3pt}{$\sim$}}}{}
     \ifthenelse{\equal{#1}{t}}{\setbox0=\hbox{$\sim$} \raisebox{-.6ex}{\hspace*{-.05ex}\adjustbox{width=#3,height=\height}{\clipbox{0.75 0 0 0}{\usebox0}}}}{}
     \ifthenelse{\equal{#1}{s}}{\rule[-0.5#2]{#3}{#2}}{}
     \ifthenelse{\equal{#1}{d}}{\setlength{\lengthvar}{#2}
            \addtolength{\lengthvar}{0.5#4}
            \rule[-\lengthvar]{#3}{#2}
            \hspace{-#3}
            \rule[0.5#4]{#3}{#2}}{}
    \ifthenelse{\equal{#1}{m}}{\setlength{\lengthvar}{1.5#2}
        \addtolength{\lengthvar}{#4}
        \rule[-\lengthvar]{#3}{#2}
        \hspace{-#3}
        \rule[-0.5#2]{#3}{#2}
        \hspace{-#3}
        \setlength{\lengthvar}{0.5#2}
        \addtolength{\lengthvar}{#4}
        \rule[\lengthvar]{#3}{#2}}{}}

\renewcommand{\turnstile}[6][s]
    {\ifthenelse{\equal{#1}{d}}
        {\sbox{\first}{$\displaystyle{#4}$}
        \sbox{\second}{$\displaystyle{#5}$}}{}
    \ifthenelse{\equal{#1}{t}}
        {\sbox{\first}{$\textstyle{#4}$}
        \sbox{\second}{$\textstyle{#5}$}}{}
    \ifthenelse{\equal{#1}{s}}
        {\sbox{\first}{$\scriptstyle{#4}$}
        \sbox{\second}{$\scriptstyle{#5}$}}{}
    \ifthenelse{\equal{#1}{ss}}
        {\sbox{\first}{$\scriptscriptstyle{#4}$}
        \sbox{\second}{$\scriptscriptstyle{#5}$}}{}
    \setlength{\dashthickness}{0.111ex}
    \setlength{\ddashthickness}{0.35ex}
    \setlength{\leasturnstilewidth}{0.8em}
    \setlength{\extrawidth}{0.2em}
    \ifthenelse{%
      \equal{#3}{n}}{\setlength{\tinyverdistance}{0ex}}{}
    \ifthenelse{%
      \equal{#3}{s}}{\setlength{\tinyverdistance}{0.5\dashthickness}}{}
    \ifthenelse{%
      \equal{#3}{d}}{\setlength{\tinyverdistance}{0.5\ddashthickness}
        \addtolength{\tinyverdistance}{\dashthickness}}{}
    \ifthenelse{%
      \equal{#3}{t}}{\setlength{\tinyverdistance}{1.5\dashthickness}
        \addtolength{\tinyverdistance}{\ddashthickness}}{}
        \setlength{\verdistance}{0.4ex}
        \settoheight{\lengthvar}{\usebox{\first}}
        \setlength{\raisedown}{-\lengthvar}
        \addtolength{\raisedown}{-\tinyverdistance}
        \addtolength{\raisedown}{-\verdistance}
        \settodepth{\raiseup}{\usebox{\second}}
        \addtolength{\raiseup}{\tinyverdistance}
        \addtolength{\raiseup}{\verdistance}
        \setlength{\lift}{0.8ex}
        \settowidth{\firstwidth}{\usebox{\first}}
        \settowidth{\secondwidth}{\usebox{\second}}
        \ifthenelse{\lengthtest{\firstwidth = 0ex}
            \and
            \lengthtest{\secondwidth = 0ex}}
                {\setlength{\turnstilewidth}{\leasturnstilewidth}}
                {\setlength{\turnstilewidth}{2\extrawidth}
        \ifthenelse{\lengthtest{\firstwidth < \secondwidth}}
            {\addtolength{\turnstilewidth}{\secondwidth}}
            {\addtolength{\turnstilewidth}{\firstwidth}}}
    \setlength{\turnstileheight}{2ex}
    \sbox{\turnstilebox}
    {\raisebox{\lift}{\ensuremath{
        \makever{#2}{\dashthickness}{\turnstileheight}{\ddashthickness}
        \makehor{#3}{\dashthickness}{\turnstilewidth}{\ddashthickness}
        \hspace{-\turnstilewidth}
        \raisebox{\raisedown}
        {\makebox[\turnstilewidth]{\usebox{\first}}}
            \hspace{-\turnstilewidth}
            \raisebox{\raiseup}
            {\makebox[\turnstilewidth]{\usebox{\second}}}
        \makever{#6}{\dashthickness}{\turnstileheight}{\ddashthickness}}}}
        \mathrel{\usebox{\turnstilebox}}}

\providecommand*{\cupdot}{%
  \mathbin{%
    \mathpalette\@cupdot{}%
  }%
}
\newcommand*{\@cupdot}[2]{%
  \ooalign{%
    $\m@th#1\cup$\cr
    \hidewidth$\m@th#1\cdot$\hidewidth
  }%
}

\makeatother

\addto\captionsamerican{\renewcommand{\definitionname}{Definition}}
\addto\captionsamerican{\renewcommand{\lemmaname}{Lemma}}
\addto\captionsamerican{\renewcommand{\notationname}{Notation}}
\addto\captionsamerican{\renewcommand{\propositionname}{Proposition}}
\addto\captionsamerican{\renewcommand{\theoremname}{Theorem}}
\addto\captionsenglish{\renewcommand{\definitionname}{Definition}}
\addto\captionsenglish{\renewcommand{\lemmaname}{Lemma}}
\addto\captionsenglish{\renewcommand{\notationname}{Notation}}
\addto\captionsenglish{\renewcommand{\propositionname}{Proposition}}
\addto\captionsenglish{\renewcommand{\theoremname}{Theorem}}
\providecommand{\definitionname}{Definition}
\providecommand{\lemmaname}{Lemma}
\providecommand{\notationname}{Notation}
\providecommand{\propositionname}{Proposition}
\providecommand{\theoremname}{Theorem}

\begin{document}

\global\long\def\fCenter{\proofsep}
	
	\global\long\def\binaryprimitive#1#2{\BinaryInf$#1\fCenter#2$}
	
	\global\long\def\axiomprimitive#1#2{\Axiom$#1\fCenter#2$}
	
	\global\long\def\unaryprimitive#1#2{\UnaryInf$#1\fCenter#2$}
	
	\global\long\def\trinaryprimitive#1#2{\TrinaryInf$#1\fCenter#2$}
	
	\global\long\def\bussproof#1{#1\DisplayProof}
	
	\global\long\def\binaryinfc#1#2#3#4{#1#2\RightLabel{\ensuremath{{\scriptstyle \textrm{#4}}}}\BinaryInfC{\ensuremath{#3}}}
	
	\global\long\def\trinaryinfc#1#2#3#4#5{#1#2#3\RightLabel{\ensuremath{{\scriptstyle \textrm{#5}}}}\TrinaryInfC{\ensuremath{#4}}}
	
	\global\long\def\unaryinfc#1#2#3{#1\RightLabel{\ensuremath{{\scriptstyle \textrm{#3}}}}\UnaryInfC{\ensuremath{#2}}}
	
	\global\long\def\axiomc#1{\AxiomC{\ensuremath{#1}}}
	
	\global\long\def\binaryinf#1#2#3#4#5{#1#2\RightLabel{\ensuremath{{\scriptstyle \textrm{#5}}}}\binaryprimitive{#3}{#4}}
	
	\global\long\def\trinaryinf#1#2#3#4#5#6{#1#2#3\RightLabel{\ensuremath{{\scriptstyle \textrm{#6}}}}\trinaryprimitive{#4}{#5}}
	
	\global\long\def\unaryinf#1#2#3#4{#1\RightLabel{\ensuremath{{\scriptstyle \textrm{#4}}}}\unaryprimitive{#2}{#3}}
	
	\global\long\def\axiom#1#2{\axiomprimitive{#1}{#2}}
	
	\global\long\def\axrulesp#1#2#3#4{\bussproof{\unaryinf{\axiomc{\vphantom{#4}}}{#1}{#2}{#3}}}
	
	\global\long\def\axrule#1#2#3{\axrulesp{#1}{#2}{#3}{\proofsep\Gamma,\Delta}}
	
	\global\long\def\unrule#1#2#3#4#5{\bussproof{\unaryinf{\axiom{#1}{#2}}{#3}{#4}{#5}}}
	
	\global\long\def\unrulec#1#2#3{\bussproof{\unaryinfc{\axiomc{#1}}{#2}{#3}}}
	
	\global\long\def\binrule#1#2#3#4#5#6#7{\bussproof{\binaryinf{\axiom{#1}{#2}}{\axiom{#3}{#4}}{#5}{#6}{#7}}}

\global\long\def\proofsep{\Rightarrow}

\title{First-Order Implication-Space Semantics}
\author{Ulf Hlobil}
\date{\today}

\maketitle

\begin{abstract}
\noindent This paper extends implication-space semantics to include
first-order quantification. Implication-space semantics has recently
been introduced as an inferentialist formal semantics that can capture
nonmonotonic and nontransitive material inferences. Extant versions,
however, include only propositional logic. This paper extends the
framework so as to recover classical first-order logic. The goal is
to formulate a theory in which consequence relations can be nonmonotonic
and supraclassical, while obeying the deduction-detachment theorem
and disjunction simplification, while also including conjunctions
that behave multiplicatively as premises and counterexamples to the
usual quantifier rules. The paper explains these constraints and shows
how they can be met jointly. The result is a first-order version of
implication-space semantics that has all the virtues for which inferentialists
and inferential expressivists praise propositional implication-space
semantics.

\emph{Keywords}:  implication-space semantics; inferentialism;
substructural logic; defeasible reasoning; material inferences; logical
expressivism
\end{abstract}

\section{Introduction}

Implication-space semantics has recently been developed as a novel
kind of inferentialist formal semantics \citep{Kaplan2022,Hlobil2024-HLORFL}.
It is inspired by, and (in some respects) similar to, Girard's \citeyearpar{GIRARD19871}
phase-space semantics for linear logic. Unlike Girard's phase-space
semantics, however, implication-space semantics can model nontransitive
and nonmonotonic material inferences. Brandom calls implication-space
semantics ``the current state of the art in inferentialist semantics''
\citep[18]{Hlobil2024-HLORFL}. And Peregrin has used a version of
this semantics to argue that inferential roles are compositional \citep{Peregrin2025}.
Classical propositional logic can be recovered in implication-space
semantics, as can the propositional fragments of Priest's Logic of
Paradox ($\mathsf{LP}$), strong Kleene logic ($\mathsf{K3}$), strict/tolerant
logic ($\mathsf{ST}$), and, with some adjustments, multiplicative-additive
linear logic ($\mathsf{MALL}$). There is, however, no extant account
of quantification in implication-space semantics. The aim of this
paper is to fill this lacuna.

The paper is structured as follows: In Section~2, I will provide
some background and explain my aim. This aim is to formulate a semantics
that can validate constraints that I label MOF, SCL, DDT, DS, and
LC without underwriting the usual quantifier rules $\forall R$ and
$\forall L$ that I will state below. In Section~3, I formulate first-order
implication-space semantics. I show that this version of implications
space semantics can validate MOF, DDT, DS, LC and need not underwrite
the rules $\forall R$ and $\forall L$. The goal of the remainder
of the paper is to show that SCL can also be met and, hence, that
classical first-order logic can be recovered in implication-space
semantics. That is the job of Section~4, in which I first put Tarskian
model theory on the table and then map it into implication-space semantics,
thus showing how to ensure that the resulting consequence relations
are supraclassical. Section~5 concludes.

\section{Background and Aim}

Since implication-space semantics is probably unfamiliar to most readers,
I start, in this section, by explaining some of the goals and challenges.
This will yield seven constraints that first-order implication-space
semantics ought to satisfy.

The goal of implication-space semantics is to provide an account of
the kind of implications that we exploit in defeasible, material (i.e.,
not logically valid) reasoning. Implications of this sort are common
not only in everyday reasoning but also in highly specialized contexts
like medical diagnosis and legal reasoning, as the following (admittedly
simplistic) examples illustrate.

\begin{elabeling}{00.00.0000}
\item [{Inf-1:}] NN is a middle-aged man with chest pain radiating to the
left arm. So, NN probably has a heart attack.
\item [{Inf-2:}] NN is a middle-aged man with chest pain radiating to the
left arm. NN's ECG and cardiac enzymes are normal, and the pain worsens
with inspiration and movement. So, NN probably has a heart attack.
\item [{Inf-3:}] A signed contract exists, according to which NN must pay
MM \$1000. So, NN must pay MM \$1000.
\item [{Inf-4:}] A signed contract exists, according to which NN must pay
MM \$1000. This contract was signed by NN under duress. So, NN must
pay MM \$1000.
\end{elabeling}

Plausibly, Inf-1 is a good inference but Inf-2 is not. And Inf-3 is
a good inference but Inf-4 is not. So, Inf-1 and Inf-3 are defeasible,
materially good inferences. What I mean by ``good inference'' here
is that the conclusion follows from the premises in some suitably
broad way, which is not always a matter of logical form. This implication
relation is not classical logical consequence and it is nonmonotonic.
Such examples can easily be multiplied, and they arise in almost all
areas of discourse, with some exceptions like mathematics. Inferentialists
like Brandom \citeyearpar{Brandom1994,Brandom2000,Brandom2008} hold
that at least some implications of this sort are content determining.
Hence, an inferentialist formal semantics must be able to capture
such implications. This is the goal of implication-space semantics.

Of course, the goodness of Inf-1 and Inf-3 are a medical and a legal
issue, respectively, and cannot be decided by any formal theory. The
goal of implication-space semantics is not to decide whether such
material implications hold but merely to provide a formal framework
in which such implications can be included (by fiat) together with
an account of logical vocabulary. The resulting theory should account
systematically for how logical and nonlogical expressions combine
into good implications by assigning contents to expressions, where
these contents are thought of as roles in such implication relations.

This goal motivates a first constraint on implication-space semantics:
it must be nonmonotonic. I will take the relevant kind of implication
to be a relation between sets of sentences,\footnote{Some have argued that inferentialists should take conclusions to be
individual sentences and not sets \citep{Steinberger2011-STEWCS-2}.
I disagree but I won't engage with this issue here for two reasons:
(i) advocates of implication-space semantics are comfortable with
a multiple-conclusion setting, and (ii) we can take the fragment of
the resulting consequence relation with just one conclusion as the
``real'' consequence relation, as I am here interested in the extension
of the consequence relation and not in any rules that might generate
that relation.} and I will express it by ``$\sttstile{}{}$'' using upper case
Greek letters for (schematic) sets of sentences and lower case Greek
letters for (schematic) sentences. We can hence formulate our first
constraint thus:

\begin{itemize}
\item \emph{Monotonicity Failures} (MOF): It can happen that $\Gamma\sttstile{}{}\Delta$
but not $\Gamma,\phi\sttstile{}{}\Delta$.
\end{itemize}

In propositional implication-space semantics, it is easy to validate
all inferences that are valid in classical propositional logic. I
will treat it as a second constraint that my version of implication-space
semantics can, similarly, validate classical first-order logic.

\begin{itemize}
\item \emph{Supraclassicality} (SCL): If $\Gamma\vdash\Delta$ holds in
classical logic, then $\Gamma\sttstile{}{}\Delta$.
\end{itemize}

In addition to these two constraints, advocates of implication-space
semantics usually endorse logical expressivism in addition to inferentialism
\citep{Brandom2000,Hlobil2024-HLORFL}. In particular, they think
that the conditional makes explicit what follows from what, so that
``If $\phi$, then $\psi$'' is implied by some premises $\Gamma$
just in case $\Gamma$ together with $\phi$ implies $\psi$. Similarly,
they hold that it something follows from a disjunction it must follow
from the disjuncts separately, and that the use of conjuncts as premises
amounts to the same as using the conjuncts as separate premises. This
gives us three further constraints:\footnote{As inferentialists, advocates of implication-space semantics typically
also have specific views about the inferential role of conditionals
as premises and the role of disjunctions and conjunctions as conclusions.
Since the connectives I will define below will meet these other desiderata
as well and the constraints here are useful to highlight why my aim
is not trivial, I do not mention this here.}

\begin{itemize}
\item \emph{Deduction-Detachment Theorem }(DDT): $\Gamma\sttstile{}{}\phi\rightarrow\psi,\Delta$
iff $\Gamma,\phi\sttstile{}{}\psi,\Delta$.
\item \emph{Disjunction Simplification }(DS): If $\Gamma,\phi\lor\psi\sttstile{}{}\Delta$,
then $\Gamma,\phi\sttstile{}{}\Delta$ and $\Gamma,\psi\sttstile{}{}\Delta$.
\item \emph{Left Conjunction }(LC): $\Gamma,\phi\wedge\psi\sttstile{}{}\Delta$
iff $\Gamma,\phi,\psi\sttstile{}{}\Delta$.
\end{itemize}

These constraints have some intuitive plausibility. Regarding DS,
for instance, if you argue from ``It is either Tuesday or Wednesday''
to ``The post office is open,'' it seems like an appropriate response
to say: ``No, the post office is closed on Wednesdays.'' And, regarding
LC, it would be implausible to take ``Tweety is a penguin'' to defeat
the inference from ``Tweety is a bird'' to ``Tweety can fly''
while holding that ``Tweety is a bird and a penguin; so Tweety can
fly'' is a good inference. However, it is not my goal to justify
these constraints; I merely state and accept them as constraints on
my project here.

These constraints have at least three important consequences. First,
SCL implies that the failures of monotonicity that implication-space
semantics aims to capture are \emph{not} the kind of monotonicity
failures that arise in relevance logics or the like. They are more
like the failures of monotonicity that default logics and preferential
logics aim to capture. Second, MOF, SCL, and DDT jointly imply that
multiplicative Cut fails in implication-space semantics. For suppose,
by MOF, that $\Gamma\sttstile{}{}\psi,\Delta$ and not $\Gamma,\phi\sttstile{}{}\psi,\Delta$.
Now, SCL implies $\psi\sttstile{}{}\phi\rightarrow\psi$. By multiplicative
Cut, $\Gamma\sttstile{}{}\phi\rightarrow\psi,\Delta$. And DDT then
yields $\Gamma,\phi\sttstile{}{}\psi,\Delta$, contradicting our assumption
\citep{Arieli2000}. Hence, implication-space semantics must be able
to accommodate nontransitive consequence relations. Third, MOF, LC,
and DS jointly imply that implication-space semantics is hyperintensional
\citep{Fine1975-FINCN-2}. For suppose, by MOF, that $\Gamma,\psi\sttstile{}{}\Delta$
and not $\Gamma,\psi,\phi\sttstile{}{}\Delta$. Now, $\psi$ is classically
equivalent to $\psi\wedge(\phi\lor\neg\phi)$; so that, in a non-hyperintensional
setting, these sentences can be substituted for each other without
affecting implications. Hence, $\Gamma,\psi\wedge(\phi\lor\neg\phi)\sttstile{}{}\Delta$.
But then, by LC, $\Gamma,\psi,(\phi\lor\neg\phi)\sttstile{}{}\Delta$;
and DS then yields $\Gamma,\psi,\phi\sttstile{}{}\Delta$, which contradicts
our assumption.

This means that no merely intensional theory like standard possible
world semantics can meet our constraints. Nor can we turn to subclassical
logics like relevance or linear logic to allow for failures of monotonicity.
Nonmonotonic logics that are supraclassical, like the most familiar
versions of default logic or preferential logics are usually, firstly,
transitive and hence fail DDT and, secondly, not hyperintensional
and hence fail either DS or LC. Finally, the best know nontransitive
logic, namely strict/tolerant logic ($\mathsf{ST}$), is monotonic
\citep{Cobreros2012}.\footnote{I do not claim that there is no way to adapt some of these logics
to meet the constraints above but merely that doing so will require
substantial and nontrivial adjustments. And implication-space semantics
can be seen as an adjustment of strict/tolerant logic, of truth-maker
semantics, or of phase-space semantics. I will not investigate such
issues here.} So giving an account of consequence relations that meet the five
constraints above is nontrivial. Extant versions of implication-space
semantics meet these constraints for propositional logic \citep{Kaplan2022,Hlobil2024-HLORFL}.
It is an open problem, however, whether---and if so, how---quantifiers
can be added to implication-space semantics, while meeting the constraints
above.

Extending implication-space semantics to a first-order language is
nontrivial. To see this, note that the usual proof rules for quantifiers
are implausible once we allow for defeasible material inferences.
Here are sequent calculus versions of these usual rules for the universal
quantifier,\footnote{I will treat the existential quantifier as defined by the dual of
the universal quantifier throughout this paper. Hence, the issues
that I am discussing have dual versions for the existential quantifier.} where we require that $y$ does not occur free in the bottom sequents
of these rules and where $[y/x]\phi$ denotes the sentence that is
like $\phi$ except that $x$ is replaced by $y$ in all free occurrences.

\medskip{}
\[
\bussproof{\unaryinf{\axiom{\Gamma,[t/x]\phi}{\Delta}}{\Gamma,\forall x\phi}{\Delta}{{\ensuremath{\forall}}L}}\qquad\bussproof{\unaryinf{\axiom{\Gamma}{[y/x]\phi,\Delta}}{\Gamma}{\forall x\phi,\Delta}{{\ensuremath{\forall}}R}}
\]

\medskip{}

One problem with the rule $\forall R$ is that it contains open formulae
since $y$ occurs free.\footnote{Unless, of course, the rule is used in a degenerate case in which
the quantifier in the bottom sequent binds a variable that doesn't
occur in its scope.} However, we don't use open formulae in natural language. Therefore
the rule cannot be applied in any straightforward way to the ordinary
language reasoning that is the target of implication-space semantics.
One option for adjusting the rule is to use a name that does not occur
in the bottom sequent instead of the variable $y$. Let us consider
a thus adjusted version of $\forall R$.

If we allow for defeasible inferences that may be good inferences
without explicitly including premises that state background assumptions,
then Inf-5 is plausibly good and, nevertheless, Inf-6 is not good.

\begin{elabeling}{00.00.0000}
\item [{Inf-5:}] \emph{Romeo and Juliette} is one of the greatest works
of English literature. So, Shakespeare is an important English author. 
\item [{Inf-6:}] \emph{Romeo and Juliette} is one of the greatest works
of English literature. So, everyone is an important English author.
\end{elabeling}

Of course, if we add as a premise that Shakespeare is the author of
\emph{Romeo and Juliette}, then the move from Inf-5 to Inf-6 violates
the constraint that the name into whose position we quantify must
not occur in the bottom sequent of our adjusted version of $\forall R$.
If we require, however, that this premise be added to Inf-5 to make
it a genuinely good inference, then this raises the question what
the general requirement for adding premises in this way is. If we
want to allow that defeasible inferences may be good inferences, then
we cannot require that our premises make the inference logically valid.
It is thus unclear how we could fix this rule so as to apply to defeasible
material inferences in a way that is not \emph{ad hoc}.

The rule $\forall L$ gives rise to a different problem. Since it
is unusual for empty beer bottles to be in fridges, Inf-7 is plausibly
a good inference; but Inf-8 is not a good inference

\begin{elabeling}{00.00.0000}
\item [{Inf-7:}] There are two beer bottles in the fridge, namely Bottle-1
and Bottle-2. If Bottle-1 is a bottle in the fridge, then Bottle-1
is empty. So, there is still beer in the fridge.
\item [{Inf-8:}] There are two beer bottles in the fridge, namely Bottle-1
and Bottle-2. For everything, if it is a bottle in the fridge, then
it is empty. So, there is still beer in the fridge.
\end{elabeling}

This example shows that, contrary to $\forall L$, in our ordinary
material and defeasible reasoning, an instance of a universal generalization
can imply a conclusion, while the universal generalization does not
imply that conclusion. The explanation is clear: the universal generalization
makes a claim about more cases than the particular instance, and the
inference at hand can be defeated by one of the cases that the universal
generalization affirms. In our case, $\forall L$ allows us to achieve
the same effect as weakening Inf-2 with the claim that Bottle-2 is
empty; but that additional premise defeats the inference.

The upshot of all this is that an extension of implication-space semantics
to include quantifiers must provide a system that meets the five constraints
above and that does not use semantic clauses for the universal quantifier
that underwrite rules like $\forall R$ and $\forall L$, so that
Inf-5 and Inf-7 may come out good while Inf-6 and Inf-8 do not come
out good, according to the system. The goal of this paper is to provide
such a system.

Of course, some readers may want to claim that logic is not meant
to apply to pieces of natural language reasoning like those above.
Note, however, that the inferentialist theory that has led to the
development of implication-space semantics clearly aspires to cover
such defeasible natural language reasoning. I will not offer any defense
of this aspiration, nor will I argue for the constraints above. My
goal is to solve the problem within the constraints just presented,
whether or not these constraints are ultimately the philosophically
correct constraints to impose.

My aim in this paper is thus to formulate a semantics that meets the
constraints MOF, SCL, DDT, DS, and LC while also tolerating counterexamples
to $\forall R$ and $\forall L$. In the next section, I will formulate
such a theory and show that it can meet MOF, DDT, DS, and LC while
also tolerating counterexamples to $\forall R$ and $\forall L$.
The remainder of the paper will then be dedicated to showing that
the theory can also meet SCL.

\section{First-Order Implication-Space Semantics}

In this section, I will present a version of implication-space semantics
for first-order languages. In the first subsection, I explain implication
frames, which I then use in the second subsection to define models
and consequence in sets of models. In the third subsection, I show
that this version of implication-space semantics satisfies MOF, DDT,
DS, and LC in the same way as extant versions of implication-space
semantics for propositional languages satisfy these constraints. Moreover,
it does not underwrite $\forall R$ and $\forall L$.

\subsection{First-Order Implication Spaces}

As in propositional implication-space semantics \citep{Hlobil2024-HLORFL},
we start with the bearers of implicational roles (``bearers'' for
short), from which we will then abstract those implicational roles.
Unlike in propositional implication-space semantics, however, these
bearers are composites. We think of them as composed of properties
and (sequences of) objects.\footnote{In contrast to Tarskian model theory, even an implication-space model
in which the set of objects is empty can be a counterexample to an
inference.} We thus start with:

\begin{itemize}
\item a set, $\mathbb{P}$, of properties, with $\mathbb{P}^{n}$ being
the set of \emph{n}-ary properties,
\item a set, $\mathbb{O}$, of objects, with $\mathbb{O}^{n}$ being the
set of all sequences of length $n$.
\end{itemize}

\begin{defn}[Bearers, $\mathbb{B}$]
 The set of bearers, $\mathbb{B}$, is the set of pairs $\left\langle \mathsf{P},\overrightarrow{\mathsf{o}}\right\rangle $,
also written $\mathsf{P}\overrightarrow{\mathsf{o}}$, for all $\mathsf{P}\in\mathbb{P}^{n}$
and $\overrightarrow{\mathsf{o}}\in\mathbb{O}^{n}$.
\end{defn}

Since implication-space semantics is motivated by inferentialism and
inferentialism is sometimes viewed as hostile to appeals to what language
represents, an opponent might worry that my starting point is in tension
with the inferentialist motivation behind implication-space semantics.
I am setting such philosophical issues aside here,\footnote{Notice that we can think of properties as predicates and objects as
names. Then the set of bearers will be the set of atomic sentences
formed from these predicates and names. Inferentialists may then take
the implication relation over these bearers to be instituted by social
norms. Nothing in what I say precludes this approach. Such philosophical
issues are, however, orthogonal to my goals in this paper.} and I will focus on the task of formulating a version of implication-space
semantics that meets the constraints set out in the introduction.

Once we have bearers of implicational roles, we can proceed in the
same way as in propositional implication-space semantics \citep{Hlobil2024-HLORFL},
up to the point where we will need a semantic clause for the universal
quantifier. An implication space is a set of all candidate implications
among sets of bearers of implicational roles.

\begin{defn}[Implication space, $\mathbb{S}$]
 $\mathbb{S}=\mathcal{P}(\mathbb{B})\times\mathcal{P}(\mathbb{B})$.
\end{defn}

An implication frame is a pair of bearers and an implication relation
among sets of these bearers. Intuitively, the implications in this
set are the good implications.\footnote{The role of this partition of all candidate implications into the
good and the bad ones is broadly similar to the role of assigning
truth-values to atomic sentences in a classical framework. Implication
frames are in implication-space semantics what phase spaces are in
Girard's \citeyearpar{GIRARD19871} phase-space semantics.}

\begin{defn}[Implication and implication frames,  $\left\langle \mathbb{B},\mathbb{I}\right\rangle $]
 Implication is a relation among sets of bearers, $\mathbb{I}\subseteq\mathbb{S}$.
An implication space is a pair $\left\langle \mathbb{B},\mathbb{I}\right\rangle $
of set of bearers, $\mathbb{B}$, and an implication relation among
sets of them.
\end{defn}

One can think of good implications philosophically in different ways.
One option is to think of bearers as similar to states of affairs
and to think of the bearers in the first element of the pair as being
made the case or obtaining and the bearers in the second element as
being made not the case or failing to obtain. Then an implication
is good if it is impossible that all the states of affairs in the
first set obtain and all the states of affairs in the second set do
not obtain. Alternatively, one can think of implications in normative-pragmatic
terms and take bearers to be sentences. In this case, one can adopt
a bilateralist idea and say that an implication holds just in case
it is normatively out of bounds to assert all the bearers in the first
set and deny all the bearers in the second set \citep[see][]{Restall2005-RESMC}.
I am setting such philosophical issues aside here.
\begin{notation}
To reduce clutter, we adopt the following notational conventions:
\begin{itemize}
\item Bearers and sets thereof (sans serif): $\mathsf{a,b,c,...A,B,C,D,...,G,...}$
\item Sets of candidate implications (sans serif): $\mathsf{H,I,J},...$ 
\item Roles and contents and sets thereof (type writer): $\mathtt{a,b,c,...A,B,C,...,G,D,...}$
\item Sentences and sets thereof (Greek letters): $\phi,\psi,...$ $\Gamma,\Delta,...$
\end{itemize}
\end{notation}

The range of subjunctive robustness of a set of candidate implications
is the set of candidate implications such that each one of them yields
a good implication when combined, by pairwise union, with any implication
in the set.

\begin{defn}[Range of subjunctive robustness, $\mathsf{RSR}(\cdotp)$]
 If $\mathsf{H}\subseteq\mathbb{S}$, then $\mathsf{RSR}(\mathsf{H})=$
$\{\left\langle x,y\right\rangle \mid\forall\left\langle \mathsf{G},\mathsf{D}\right\rangle \in\mathsf{H}\thinspace(\left\langle \mathsf{G}\cup x,\mathsf{D}\cup y\right\rangle \in\mathbb{I})\}$.
\end{defn}

We now define implicational roles in terms of ranges of subjunctive
robustness. Bearers, candidate implications, and sets thereof share
their implicational roles just in case they have the same range of
subjunctive robustness. Thus, implicational roles are equivalence
classes, with the equivalence relation being sameness of the range
of subjunctive robustness.

\begin{defn}[Implicational Role, $\mathcal{R}(\cdot)$]
 Bearers, candidate implications, and sets thereof have the following
implicational roles, written $\mathcal{R}(\cdot)$:
\begin{elabeling}{00.0}
\item [{(1)}] If $\mathsf{H}\subseteq\mathbb{S}$, then $\mathcal{R}(\mathsf{H})$
$=\{x\mid\mathsf{RSR}(\mathsf{H})=\mathsf{RSR}(x)\}$.\vspace{-6pt}
\item [{(2)}] If $\mathsf{A}\in\mathbb{S}$, then $\mathcal{R}(\mathsf{A})$
$=\mathcal{R}(\{\mathsf{A}\})$.\vspace{-6pt}
\item [{(3)}] If $\mathsf{a}\in\mathbb{B}$, then $\mathcal{R}(\mathsf{a})=$
$\left\langle \mathcal{R}^{+}(\mathsf{a}),\mathcal{R}^{-}(\mathsf{a})\right\rangle $,
where $\mathcal{R}^{+}(\mathsf{a})$ $=\mathcal{R}\left\langle \{\mathsf{a}\},\emptyset\right\rangle $
is the premisory role of $\mathsf{a}$ and $\mathcal{R}^{-}(\mathsf{a})$
$=\mathcal{R}\left\langle \emptyset,\{\mathsf{a}\}\right\rangle $
is the conclusory role of $\mathsf{a}$.
\end{elabeling}
\end{defn}

\noindent By the notational convention above, we use typewriter font
for roles; so: $\mathtt{P}\overrightarrow{\mathtt{o}}=\mathcal{R}(\mathsf{P}\overrightarrow{\mathsf{o}})$,
$\mathtt{P}\overrightarrow{\mathtt{o}}^{+}=\mathcal{R}^{+}(\mathsf{P}\overrightarrow{\mathsf{o}})$,
$\mathtt{P}\overrightarrow{\mathtt{o}}^{-}=\mathcal{R}^{-}(\mathsf{P}\overrightarrow{\mathsf{o}})$,
$\mathtt{H}=\mathcal{R}(\mathsf{H})$, etc.

Notice that $\bigcup\mathcal{R}(\mathsf{H})\in\mathcal{R}(\mathsf{H})$,
for any $\mathsf{H}$. For, all $\mathsf{J}\in\mathcal{R}(\mathsf{H})$
have the same range of subjunctive robustness, $\mathsf{RSR}(\mathsf{J})=\mathsf{RSR}(\mathsf{H})$;
and this means that the union of these sets also has this range of
subjunctive robustness: $\mathsf{RSR}(\mathsf{J})=\mathsf{RSR}(\bigcup\mathcal{R}(\mathsf{H}))$.
We can thus treat the union of a role's elements as the canonical
representative of the role. This gives us a representative member
of a role whenever needed, even when we are not explicitly given any
representative of this equivalence class.

Below I will interpret sentences by assigning pairs of implicational
roles to them: one role as a premise (premisory role), and one role
as a conclusion (conclusory role). I call such pairs of roles conceptual
contents. Being maximally liberal regarding what conceptual contents
are, I let any pair of implicational roles count as a conceptual content.

\begin{defn}[Conceptual Content]
 If there are two implicational roles $\mathcal{\mathtt{a}}^{+}=\mathcal{R}(\mathsf{H})$
and $\mathcal{\mathtt{a}}^{-}=\mathcal{R}(\mathsf{G})$ (for some
sets of candidate implications $\mathsf{H}$ and $\mathsf{G}$), then
the pair of them is a conceptual content, $\mathtt{a}$. That is,
$\mathtt{a}=\left\langle \mathcal{\mathtt{a}}^{+},\mathtt{a}^{-}\right\rangle $,
where $\mathcal{\mathtt{a}}^{+}$ is the premisory role and $\mathcal{\mathtt{a}}^{-}$
is the conclusory role of the content. The collection of all contents
is called $\mathbb{C}$.
\end{defn}

I will later assign conceptual contents to complex sentences in terms
of the contents of their constituents. To this end, I now define some
operations on roles, namely adjunction, symjunction, and power-symjunction.

\begin{defn}[Adjunction, $\sqcup$]
 Let $\mathsf{H},\mathsf{J}\subseteq\mathbb{S}$, then: $\mathcal{R}(\mathsf{H})\sqcup\mathcal{R}(\mathsf{J})=$
$\mathcal{R}(\{\left\langle \mathsf{A}\cup\mathsf{C},\mathsf{B}\cup\mathsf{D}\right\rangle |\left\langle \mathsf{A},\mathsf{B}\right\rangle \in\mathsf{H},\left\langle \mathsf{C},\mathsf{D}\right\rangle \in\mathsf{J}\})$.
If $\mathtt{R}$ is a set of roles, then we write $\bigsqcup\mathtt{R}$
for the adjunction of the roles in $\mathtt{R}$.
\end{defn}

\begin{defn}[Symjunction, $\sqcap$]
 Let $\mathsf{H},\mathsf{J}\subseteq\mathbb{S}$, then: $\mathcal{R}(\mathsf{H})\sqcap\mathcal{R}(\mathsf{J})$
$=\mathcal{R}(\mathsf{H}\cup\mathsf{J})$. If $\mathtt{R}$ is a set
of roles, then we write $\bigsqcap\mathtt{R}$ for the adjunction
of the roles in $\mathtt{R}$.
\end{defn}

\begin{defn}[Power-Symjunction, $\APLdown$]
 If $\mathtt{R}$ is a set of roles, then the power-symjunction of
the roles in $\mathtt{R}$, written $\APLdown\mathtt{R}$, is the
symjunction of the adjunction of every nonempty subset of $\mathtt{R}$,
that is: $\APLdown\mathtt{R}\,=$ $\bigsqcap\{\bigsqcup x\mid x\subseteq\mathtt{R},x\neq\emptyset\}$.
\end{defn}

Below I will define models in terms of implication frames. Since I
will not make use of the idea of truth-in-a-model in implication-space
semantics, however, I cannot say that an implication-space model is
a counterexample to an inference iff, in the model, all the premises
are true and all the conclusions aren't true. Instead, I now define
a notion of entailment among contents, and I will later say that a
model is a counterexample to an inference iff the set of contents
that the model assigns to the premises does not entail the set of
contents that the model assigns to the conclusions.

The idea behind the following definition of implication is that a
set of contents, $\mathtt{G}$, implies a set of contents, $\mathtt{D}$,
just in case the adjunction of all the premisory roles of $\mathtt{G}$
and all the conclusory roles of $\mathtt{D}$ is a role that contains
only good implications. That is, anything that plays the role of combining
all of the first contents as premises and all of the second contents
as conclusions always yields a good implication.

\begin{defn}[Content entailment in an implication frame, $\dttstile{}{}$]
 Given an implication frame $\left\langle \mathbb{B},\mathbb{I}\right\rangle $,
let $\mathtt{G},\mathtt{D}\subseteq\mathbb{C}$ and $\mathtt{G}=\{\mathtt{g}_{0},...,\mathtt{g}_{n}\}$
and $\mathtt{D}=\{\mathtt{d}_{0},...,\mathtt{d}_{m}\}$, then $\mathtt{G}\dttstile{}{}\mathtt{D}$
holds in the implication frame if and only if $\bigcup(\stackrel[i=0]{n}{\sqcup}\mathtt{g}_{i}^{+}\thinspace\sqcup\stackrel[j=0]{m}{\sqcup}\mathtt{d}_{j}^{-})\subseteq\mathbb{I}$.
\end{defn}

Notice that this entailment relation holds neither among sentences
nor among bearers of implicational roles but among contents. Unsurprisingly,
however, an entailment holds among the contents of two sets of bearers,
$\mathtt{G}\dttstile{}{}\mathtt{D}$, just in case the pair of sets
of bearers is a good implication in the frame, $\left\langle \mathsf{G},\mathsf{D}\right\rangle \in\mathbb{I}$.
In this sense, every implication frame interprets itself.

\begin{prop}
Let $\mathsf{G}\cup\mathsf{D}\subseteq\mathbb{B}$ and $\mathtt{G}=\{\mathcal{R}(\mathsf{g})\mid\mathsf{g}\in\mathsf{G}\}$$=\{\mathtt{g}_{0},...,\mathtt{g}_{n}\}$
and $\mathtt{D}=\{\mathcal{R}(\mathsf{d})\mid\mathsf{d}\in\mathsf{D}\}$$=\{\mathtt{d}_{0},...,\mathtt{d}_{n}\}$,
then: $\left\langle \mathsf{G},\mathsf{D}\right\rangle \in\mathbb{I}$
iff $\mathtt{G}\dttstile{}{}\mathtt{D}$.\label{prop:Frames-interpret-themselves}
\end{prop}

\begin{proof}
$\left\langle \mathsf{G},\mathsf{D}\right\rangle \in\mathbb{I}$ iff,
for every $x\in\mathcal{R}(\left\langle \mathsf{G},\mathsf{D}\right\rangle )$,
$\left\langle \emptyset,\emptyset\right\rangle \in\mathsf{RSR}(x)$.
That holds iff, for all candidate implications $y\in x$, $y\in\mathbb{I}$.
That means that $x\subseteq\mathbb{I}$, and hence that the union
of the elements of the role is also a subset of $\mathbb{I}$. So,
$\left\langle \mathsf{G},\mathsf{D}\right\rangle \in\mathbb{I}$ iff
$\bigcup\mathcal{R}(\left\langle \mathsf{G},\mathsf{D}\right\rangle )\subseteq\mathbb{I}$.
But $\mathcal{R}(\left\langle \mathsf{G},\mathsf{D}\right\rangle )$
just is $\stackrel[i=0]{n}{\sqcup}\mathtt{g}_{i}^{+}\thinspace\sqcup\stackrel[j=0]{m}{\sqcup}\mathtt{d}_{j}^{-}$.
Therefore, $\left\langle \mathsf{G},\mathsf{D}\right\rangle \in\mathbb{I}$
iff $\bigcup(\stackrel[i=0]{n}{\sqcup}\mathtt{g}_{i}^{+}\thinspace\sqcup\stackrel[j=0]{m}{\sqcup}\mathtt{d}_{j}^{-})\subseteq\mathbb{I}$,
which means, by definition, that $\mathtt{G}\dttstile{}{}\mathtt{D}$.
\end{proof}

Except for starting with bearers that are composites of properties
and objects (and the definition of power-symjunction), I have so far
merely copied everything from propositional implication-space semantics.
It will be useful below, however, to define a notion of substitution
of objects in bearers of implicational roles, as this will be helpful
in exploiting the internal structure of bearers to state the semantic
clause for the universal quantifier.

\begin{defn}[{Substitution of objects, $[\cdotp/\cdotp]\cdotp$}]
 For any bearer $\mathsf{P}\overrightarrow{\mathsf{o}}$ and objects
$\mathtt{o}_{j}$ and $\mathtt{o}_{k}$, $[\mathtt{o}_{j}/\mathtt{o}_{k}]\mathsf{P}\overrightarrow{\mathsf{o}}$
is the bearer like $\mathsf{P}\overrightarrow{\mathsf{o}}$ except
that $\mathtt{o}_{k}$ is replaced by $\mathtt{o}_{j}$ everywhere
in $\overrightarrow{\mathsf{o}}$. For a set of bearers $\mathsf{A}$,
$[\mathtt{o}_{j}/\mathtt{o}_{k}]\mathsf{A}$ is the set of $[\mathtt{o}_{j}/\mathtt{o}_{k}]\mathsf{P}\overrightarrow{\mathsf{o}}$
for all $\mathsf{P}\overrightarrow{\mathsf{o}}\in\mathsf{A}$. We
can write $[\mathtt{o}_{j}/\mathtt{o}_{k}]\mathcal{R}(\mathsf{A})$
for $\mathcal{R}([\mathtt{o}_{j}/\mathtt{o}_{k}]\mathsf{A})$.
\end{defn}

To illustrate, $[\mathtt{d}/\mathtt{c}]\mathsf{Pabc}=\mathsf{Pabd}$.
Notice that it can happen that $\mathcal{R}(\mathsf{A})=\mathcal{R}(\mathsf{B})$
but $\mathcal{R}([\mathtt{o}_{j}/\mathtt{o}_{k}]\mathsf{A})\not=\mathcal{R}([\mathtt{o}_{j}/\mathtt{o}_{k}]\mathsf{B})$.
This can happen, \emph{inter alia}, if the bearers in $\mathsf{A}$
contain object $\mathtt{o}_{k}$ but the bearers in $\mathsf{B}$
don't. So, when we write $[\mathtt{o}_{j}/\mathtt{o}_{k}]\mathcal{R}(\mathsf{A})$,
we cannot think of $\mathcal{R}(\mathsf{A})$ as a role on which $[\mathtt{o}_{j}/\mathtt{o}_{k}]$
acts. The substitution of objects is defined for bearers and not for
roles. It is nevertheless often convenient to use this notation, for
instance, to write $\APLdown\{\mathcal{R}([\mathtt{o}_{j}/\mathtt{o}_{k}]\mathsf{A}),\mathcal{R}([\mathtt{o}_{j}/\mathtt{o}_{k}]\mathsf{B})\}$
as $[\mathtt{o}_{j}/\mathtt{o}_{k}]\APLdown\{\mathcal{R}(\mathsf{A}),\mathcal{R}(\mathsf{B})\}$
or to write $\bigsqcup\{\mathcal{R}([\mathtt{o}_{j}/\mathtt{o}_{k}]\mathsf{A}),\mathcal{R}([\mathtt{o}_{j}/\mathtt{o}_{k}]\mathsf{B})\}$
as $[\mathtt{o}_{j}/\mathtt{o}_{k}]\bigsqcup\{\mathcal{R}(\mathsf{A}),\mathcal{R}(\mathsf{B})\}$.

We have now seen what implication frames are, how we can abstract
roles from these frames, and how we can define contents as pairs of
such roles. Moreover, I have defined three operations on such roles.
We can use the thus defined contents to interpret a language by assigning
to each sentence a content that we abstract from a given implication
frame.

\subsection{Interpreting First-Order Languages}

In this subsection, I will use the implication frames and contents
defined in the previous subsection to interpret the expressions of
a standard first-order language. I will work with first-order languages
that contain neither identity, nor functors, nor primitive atomic
sentences. Let me say explicitly what the languages at issue look
like.
\begin{quote}
Language $\mathcal{L}$:
\begin{itemize}
\item $\mathsf{Pred}^{n}$: $F_{1},F_{2},...$ is a countable stock of predicates,
each with a given and fixed arity $n$.
\item Terms:
\begin{itemize}
\item $\mathsf{Name}$: $a_{1},a_{2},...$ is a countable stock of names.
\item $\mathsf{Var}$: $x_{1},x_{2},...$ is a countable stock of variables.
\end{itemize}
\item $\mathsf{At_{Sent}}$: any $n$-ary predicate followed by $n$ names
is an atomic sentence, and nothing else is an atomic sentence.
\item $\mathsf{At}$: any $n$-ary predicate followed by $n$ terms is an
atomic formula, and nothing else is.
\item $\mathsf{Form}$: An atomic formula is a formula. If $\phi$ and $\psi$
are formulae and $x_{i}$ is a variable, then $\neg\phi$, $\phi\wedge\psi$,
and $\forall x_{i}\phi$ are formulae. Disjunction, conditionals,
and existential generalizations are treated as defined: $\phi\lor\psi=_{df.}\neg(\neg\phi\wedge\neg\psi)$,
$\phi\rightarrow\psi=_{df.}\neg(\phi\wedge\neg\psi)$, and $\exists x_{i}\phi=_{df.}\neg\forall x_{i}\neg\phi$.
\item $\mathsf{Sent}$: A sentence is a formula in which no variable occurs
free.
\end{itemize}
\end{quote}

In order to account for quantification, we will have to make use of
the idea of substituting terms for other terms in formulae, in the
following familiar sense.

\begin{defn}[{Substitution of terms, $[\cdotp/\cdotp]\cdotp$}]
 Terms: $[t/a](a)=t$ and $[t/a](b)=b$ if $b\neq a$. Atomic sentences:$[t/a](Ft_{k}....t_{m})=F[t/a](t_{k})....[t/a](t_{m})$.
Complex sentences: If $\circ$ is some sentential connective, then
$[t/a](\phi\circ\psi)$ $=[t/a]\phi\circ[t/a]\psi$. Moreover, $[t/x](\forall x\phi)$
$=\forall x\phi$. And $[t/a](\forall x\phi)$ $=\forall x[t/a]\phi$,
if $a\neq x$.
\end{defn}

With our language thus in place, we can turn to models. An implication-space
model consists of the implication frame, from which we abstract contents
as above, and an interpretation function that maps sentences to such
contents.

\begin{defn}[Implication-space model]
 An implication-space model, $\mathcal{M}$, is a pair $\left\langle \left\langle \mathbb{B},\mathbb{I}\right\rangle ,\left\llbracket \cdotp\right\rrbracket ^{\mathcal{M}}\right\rangle $
consisting of an implication frame, $\left\langle \mathbb{B},\mathbb{I}\right\rangle $,
and an interpretation function $\left\llbracket \cdotp\right\rrbracket ^{\mathcal{M}}$
that maps all sentences of a given language, $\mathcal{L}$, to contents
in $\mathbb{C}$ defined by $\left\langle \mathbb{B},\mathbb{I}\right\rangle $.
\end{defn}

For our case of a first-order language, we must specify further what
the interpretation function of a model looks like. This is where we
must provide semantic clauses for our logical vocabulary. For negation
and conjunction, I will use the clauses from propositional implication-space
semantics \citep{Hlobil2024-HLORFL}. The novelty lies in the clause
for the universal quantifier. Names are interpreted by objects and
predicates by properties, and the content of atomic sentences is the
role of the corresponding bearer.

\begin{defn}[Interpretation function, $\left\llbracket \cdotp\right\rrbracket $]
\label{def:NMMS-interpretations} An interpretation function $\left\llbracket \cdotp\right\rrbracket $
maps expressions of a language $\mathcal{L}$ to conceptual contents
$\mathbb{C}$, objects $\mathbb{O}$, and properties $\mathbb{P}$
of an implication frame as follows: $\forall a\in\mathsf{Name}(\left\llbracket a\right\rrbracket \in\mathbb{O})$
and $\forall F\in\mathsf{Perd}^{n}(\left\llbracket F\right\rrbracket \in\mathbb{P}^{n})$.
If $Fa_{1},...,a_{i}$ is an atomic sentence such that $\mathsf{P}=\left\llbracket F\right\rrbracket $
and $\overrightarrow{\mathsf{o}}=\left\langle \left\llbracket a_{1}\right\rrbracket ,...,\left\llbracket a_{i}\right\rrbracket \right\rangle $,
then $\left\llbracket Fa_{1},...,a_{i}\right\rrbracket =_{df.}\mathcal{R}(\mathsf{P}\overrightarrow{\mathsf{o}})$.
The interpretations of complex sentences are:
\begin{itemize}
\item $\left\llbracket \neg\phi\right\rrbracket =_{df.}\left\langle \left\llbracket \phi\right\rrbracket ^{-},\left\llbracket \phi\right\rrbracket ^{+}\right\rangle $,
\item $\left\llbracket \phi\wedge\psi\right\rrbracket =_{df.}\left\langle \bigsqcup\{\left\llbracket \phi\right\rrbracket ^{+},\left\llbracket \psi\right\rrbracket ^{+}\},\APLdown\{\left\llbracket \phi\right\rrbracket ^{-},\left\llbracket \psi\right\rrbracket ^{-}\}\right\rangle $
\item $\left\llbracket \forall x[x/a]\phi\right\rrbracket =_{df.}\left\langle \bigsqcup\{[\mathtt{o}/\left\llbracket a\right\rrbracket ]\left\llbracket \phi\right\rrbracket ^{+}\mid\mathsf{o}\in\mathbb{O}\},\APLdown\{[\mathtt{o}/\left\llbracket a\right\rrbracket ]\left\llbracket \phi\right\rrbracket ^{-}\mid\mathsf{o}\in\mathbb{O}\}\right\rangle $
\end{itemize}
Interpretations of sets of sentences are the set of the interpretants
of the sentences, that is, $\left\llbracket \Gamma\right\rrbracket =\{\left\llbracket \gamma\right\rrbracket \mid\gamma\in\Gamma\}$.
\end{defn}

As is obvious from inspecting the clauses of conjunction and universal
quantification, the idea is to treat the universal quantifier as,
roughly, a ``generalized conjunction.'' That is, if we had names
for all objects, the premisory role of a universal generalization
would be the premisory role of the conjunction of all its instances
and the conclusory role would be the conclusory role of the conjunction
of all of its instances.

If we apply the definitions of the conditional, disjunction, and existential
generalization and work through the semantic clauses above, we arrive
at the following clauses:
\begin{itemize}
\item $\left\llbracket \phi\rightarrow\psi\right\rrbracket =\left\langle \APLdown\{\left\llbracket \phi\right\rrbracket ^{-},\left\llbracket \psi\right\rrbracket ^{+}\},\bigsqcup\{\left\llbracket \phi\right\rrbracket ^{+},\left\llbracket \psi\right\rrbracket ^{-}\}\right\rangle $
\item $\left\llbracket \phi\lor\psi\right\rrbracket =_{}\left\langle \APLdown\{\left\llbracket \phi\right\rrbracket ^{+},\left\llbracket \psi\right\rrbracket ^{+}\},\bigsqcup\{\left\llbracket \phi\right\rrbracket ^{-},\left\llbracket \psi\right\rrbracket ^{-}\}\right\rangle $
\item $\left\llbracket \exists x[x/a]\phi\right\rrbracket =\left\langle \APLdown\{[\mathtt{o}/\left\llbracket a\right\rrbracket ]\left\llbracket \phi\right\rrbracket ^{+}\mid\mathsf{o}\in\mathbb{O}\},\bigsqcup\{[\mathtt{o}/\left\llbracket a\right\rrbracket ]\left\llbracket \phi\right\rrbracket ^{-}\mid\mathsf{o}\in\mathbb{O}\}\right\rangle $
\end{itemize}
Unsurprisingly, the existential quantifier is, on this account, a
``generalized disjunction'' in the same sense in which the universal
quantifier is a ``generalized conjunction.''

Since I don't use the notion of truth-in-a-model, I cannot say that
an inference is valid, relative to a set of implication-space models
iff at least one conclusion is true in every model in which all the
premises are true. Hence, we need another way to define consequence
or implication, as I will say, relative to a set of models.

\begin{defn}[Implication in $m$-models, $\sttstile{}{m}$]
 Given a subset, $m$, of all implication-space models, for a given
language, we say that the set of sentences $\Gamma$ implies the set
of sentences $\Delta$ in $m$, written $\Gamma\sttstile{}{m}\Delta$
iff, for all $\mathcal{M}\in m$, $\left\llbracket \Gamma\right\rrbracket ^{\mathcal{M}}\dttstile{}{}\left\llbracket \Delta\right\rrbracket ^{\mathcal{M}}$
in the implication frame of $\mathcal{M}$.
\end{defn}

We can now ask: Is there a class of implication-space models $m$
such that $\Gamma\sttstile{}{m}\Delta$ iff $\Gamma\models_{FOL}\Delta$,
that is, iff $\Gamma$ implies $\Delta$ in classical first-order
logic? And, if so, can this class of models be characterized in an
easy and illuminating way? The answer to both of these questions is
``yes.'' Showing this, however, will take a bit of work; and the
job of the next section is to do that work. Before we get to that
we should consider whether the just presented version of implication-space
semantics meets the other desiderata from the Introduction.

\subsection{Are the Desiderata Met?\label{subsec:Are-the-Desiderata-met}}

Recall that my goal is to formulate a version of implication-space
semantics that obeys the constraints MOF, SCL, DDT, DS, LC and doesn't
underwrite the rules $\forall R$ and $\forall L$. The reasons why
MOF, DDT, DS, and LC hold are the same as in propositional implication-space
semantics.
\begin{prop}
It can happen that $\Gamma\sttstile{}{m}\Delta$ but not $\Gamma,\phi\sttstile{}{m}\Delta$.\label{prop:MO-fails}
\end{prop}

\begin{proof}
There is an implication frame with that $\left\langle \{\mathsf{a}\},\{\mathsf{b}\}\right\rangle \in\mathbb{I}$
but not $\left\langle \{\mathsf{a},\mathsf{d}\},\{\mathsf{b}\}\right\rangle \in\mathbb{I}$.
By Proposition~\ref{prop:Frames-interpret-themselves}, it follows
that $\{\mathsf{a}\}\dttstile{}{}\{\mathsf{b}\}$ but not $\{\mathsf{a},\mathsf{d}\}\dttstile{}{}\{\mathsf{b}\}$
in that implication frame. And if we let $\left\llbracket \phi\right\rrbracket =\mathtt{a}$,
$\left\llbracket \psi\right\rrbracket =\mathtt{b}$, $\left\llbracket \chi\right\rrbracket =\mathtt{d}$
and $m$ is the singelton of such a model, then $\phi\sttstile{}{m}\psi$
but not $\phi,\chi\sttstile{}{m}\psi$.
\end{proof}

\begin{prop}
DDT, DS, and LC hold in all sets, $m$, of implication-space models:
(DDT) $\Gamma,\phi\sttstile{}{m}\psi,\Delta$ iff $\Gamma\sttstile{}{m}\phi\rightarrow\psi,\Delta$;
(DS) if $\Gamma,\phi\lor\psi\sttstile{}{m}\Delta$, then $\Gamma,\phi\sttstile{}{m}\Delta$
and $\Gamma,\psi\sttstile{}{m}\Delta$; (LC) $\Gamma,\phi\wedge\psi\sttstile{}{m}\Delta$
iff $\Gamma,\phi,\psi\sttstile{}{m}\Delta$.\label{prop:DDT-DS-LC-hold}
\end{prop}

\begin{proof}
We begin with LC. By our semantic clause for conjunction, $\left\llbracket \phi\wedge\psi\right\rrbracket ^{+}$
$=\bigsqcup\{\left\llbracket \phi\right\rrbracket ^{+},\left\llbracket \psi\right\rrbracket ^{+}\}$,
which is $\left\llbracket \phi\right\rrbracket ^{+}\sqcup\left\llbracket \psi\right\rrbracket ^{+}$.
Hence, they can be substituted salva vertiate in the place of $X$
in the schema $\bigcup(\stackrel[i=0]{n}{\sqcup}\mathtt{g}_{i}^{+}\thinspace\sqcup X\thinspace\sqcup\stackrel[j=0]{m}{\sqcup}\mathtt{d}_{j}^{-})\subseteq\mathbb{I}$.
Therefore, $\mathtt{G}\cup\{\left\llbracket \phi\wedge\psi\right\rrbracket \}\dttstile{}{}\mathtt{D}$
iff $\mathtt{G}\cup\{\left\llbracket \phi\right\rrbracket ,\left\llbracket \psi\right\rrbracket \}\dttstile{}{}\mathtt{D}$,
for any sets of contents $\mathtt{G}$ and $\mathtt{D}$, in all implication
space models. So, $\Gamma,\phi,\psi\sttstile{}{m}\Delta$ iff $\Gamma,\phi\wedge\psi\sttstile{}{m}\Delta$
in all sets of implication-space models $m$.

For DDT and DS, I will use the derived semantic clauses above. For
DDT, note that $\left\llbracket \phi\rightarrow\psi\right\rrbracket ^{-}$
is $\bigsqcup\{\left\llbracket \phi\right\rrbracket ^{+},\left\llbracket \psi\right\rrbracket ^{-}\}$.
So they can be substituted for each other in the place of $X$ in
the schema above. Therefore, $\mathtt{G}\dttstile{}{}\{\left\llbracket \phi\rightarrow\psi\right\rrbracket \}\cup\mathtt{D}$
iff $\mathtt{G}\cup\{\left\llbracket \phi\right\rrbracket \}\dttstile{}{}\{\left\llbracket \psi\right\rrbracket \}\cup\mathtt{D}$.
So, $\Gamma\sttstile{}{m}\phi\rightarrow\psi,\Delta$ iff $\Gamma,\phi\sttstile{}{m}\psi,\Delta$
in all sets of implication-space models $m$.

Regarding DS, notice that $\left\llbracket \phi\lor\psi\right\rrbracket ^{+}$
is $\left\llbracket \phi\right\rrbracket ^{+}\sqcap\left\llbracket \psi\right\rrbracket ^{+}\sqcap(\left\llbracket \phi\right\rrbracket ^{+}\sqcup\left\llbracket \psi\right\rrbracket ^{+})$.
So, if we let $\left\llbracket \phi\right\rrbracket ^{+}=\mathcal{R}(\mathsf{H})$
and $\left\llbracket \psi\right\rrbracket ^{+}=\mathcal{R}(\mathsf{J)}$
and $\left\llbracket \phi\right\rrbracket ^{+}\sqcup\left\llbracket \psi\right\rrbracket ^{+}=\mathcal{R}(\mathsf{K)}$,
then $\left\llbracket \phi\lor\psi\right\rrbracket ^{+}=\mathcal{R}(\mathsf{H}\cup\mathsf{J}\cup\mathsf{K})$.
Now the range of subjunctive robustness that every element of $\mathcal{R}(\mathsf{H}\cup\mathsf{J}\cup\mathsf{K})$
shares is the intersection of the ranges of subjunctive robustness
of $\mathsf{H}$ and $\mathsf{J}$ and $\mathsf{K}$. So $\mathsf{RSR}(\bigcup\mathcal{R}(\mathsf{H}\cup\mathsf{J}\cup\mathsf{K}))\subseteq\mathsf{RSR}(\bigcup\mathcal{R}(\mathsf{H}))$
and $\mathsf{RSR}(\bigcup\mathcal{R}(\mathsf{H}\cup\mathsf{J}\cup\mathsf{K}))\subseteq\mathsf{RSR}(\bigcup\mathcal{R}(\mathsf{J}))$.
Now, for all $\mathsf{X}$ and $\mathsf{Y}$, if $\mathsf{RSR}(\bigcup\mathcal{R}(\mathsf{X}))\subseteq\mathsf{RSR}(\bigcup\mathcal{R}(\mathsf{Y}))$,
then $\mathtt{G}\cup\{\mathcal{R}(\mathsf{X})\}\dttstile{}{}\mathtt{D}$
implies $\mathtt{G}\cup\{\mathcal{R}(\mathsf{Y})\}\dttstile{}{}\mathtt{D}$,
for every $\mathtt{G}$ and $\mathtt{D}$. So, if $\mathtt{G}\cup\{\mathcal{R}(\mathsf{H}\cup\mathsf{J}\cup\mathsf{K})\}\dttstile{}{}\mathtt{D}$,
then $\mathtt{G}\cup\{\mathcal{R}(\mathsf{H})\}\dttstile{}{}\mathtt{D}$
and $\mathtt{G}\cup\{\mathcal{R}(\mathsf{J})\}\dttstile{}{}\mathtt{D}$.
Therefore, if $\Gamma,\phi\lor\psi\sttstile{}{m}\Delta$, then $\Gamma,\phi\sttstile{}{m}\Delta$,
in all sets of models $m$.
\end{proof}

\begin{prop}
There are implication-space models that are counterexamples to $\forall R$
and $\forall L$, that is: it can happen that $\Gamma\sttstile{}{m}Fa,\Delta$,
where $a$ does not occur in any sentence in $\Gamma$ or $\Delta$,
but not $\Gamma\sttstile{}{m}\forall xFx,\Delta$ and it can happen
$\Gamma,Fa\sttstile{}{m}\Delta$ but not $\Gamma,\forall xFx\sttstile{}{m}\Delta$.\label{prop:quantifier-rules-fail}
\end{prop}

\begin{proof}
To show that the adjusted version of $\forall R$ can fail, consider
a model with just two objects, $\mathsf{o_{1}}$ and $\mathsf{o_{2}}$,
and one property, $\mathsf{F}$, such that $\left\langle \emptyset,\{\mathsf{Fo_{1}}\}\right\rangle \in\mathbb{I}$
but not $\left\langle \emptyset,\{\mathsf{Fo_{2}}\}\right\rangle \in\mathbb{I}$.
By Proposition~\ref{prop:Frames-interpret-themselves}, in such a
model, $\thinspace\dttstile{}{}\mathtt{\mathsf{Fo_{1}}}$ but not
$\thinspace\dttstile{}{}\mathtt{\mathsf{Fo_{2}}}$. Hence, not $\thinspace\dttstile{}{}\mathtt{\mathtt{\mathsf{Fo_{1}}}\sqcap\mathsf{Fo_{2}}\sqcap\mathtt{(\mathsf{Fo_{1}}}\sqcup\mathsf{Fo_{2}}})$.
Letting $m$ be the singleton of this model and $\left\llbracket Fo_{1}\right\rrbracket =\mathtt{\mathsf{Fo_{1}}}$,
we get $\thinspace\sttstile{}{m}Fo_{1}$ but not $\thinspace\sttstile{}{m}\forall xFx$.

Regarding $\forall L$, note that there is a model in which $\left\langle \{\mathsf{Fo_{1}}\},\mathtt{D}\right\rangle \in\mathbb{I}$
but not $\left\langle \{\mathsf{Fo_{1}},\mathsf{Fo_{2}}\},\mathtt{D}\right\rangle \in\mathbb{I}$.
By reasoning analogous to that for $\forall R$, letting $m$ be the
singleton of this model with $\left\llbracket \Delta\right\rrbracket =\mathtt{D}$
and $\left\llbracket Fo_{1}\right\rrbracket =\mathtt{\mathsf{Fo_{1}}}$,
we know that $Fo_{1}\sttstile{}{m}\Delta$ but not $\forall xFx\sttstile{}{m}\Delta$.
\end{proof}
For all the desiderata from the introduction except supraclassicality
(SCL), we have now seen that they can be met by the first-order implication-space
semantics above. The arguments in the introduction that showed that
it is nontrivial to meet these desiderata jointly, however, crucially
relied on SCL. It is, hence, important to show that SCL can also be
met.

\section{Recovering Classical First-Order Logic}

The goal of this section is to show how first-order logic can be recovered
in the implication-space semantics that I have presented in the previous
section. To this end, I will define a mapping of the models of standard
Tarskian model theory to the model of implication-space semantics.
I will then show that this mapping preserves to which inferences the
models are counterexamples. Next, I will characterize the class of
implication-space models that correspond to the models of standard
Tarskian model theory. In order to do all this, I first need to put
standard Tarskian model theory on the table, and I need to establish
some facts about it.

\subsection{Recasting Truth in a Tarskian Model}

A Tarskian model is a model that can be constructed from the same
basic building blocks as our implications-space models: objects and
properties. However, \emph{n}-ary properties are taken to be sets
of \emph{n}-tuples of objects, namely the set of tuples that satisfy
the property.

\begin{defn}[Tarskian model]
 $\mathcal{M}_{T}=\left\langle \mathbb{O},\left\llbracket \cdotp\right\rrbracket ^{\mathcal{M}_{T}}\right\rangle $,
i.e., a Tarskian model is a pair of a set of objects, $\mathbb{O}$,
and an interpretation function, $\left\llbracket \cdotp\right\rrbracket ^{\mathcal{M}_{T}}$,
that assigns expressions in a first-order language, $\mathcal{L}$,
interpretants constructed on the basis of $\mathbb{O}$. For any non-zero
integer, $n$, we have the set of all $n$-tuples of objects in $\mathbb{O}$,
which we denote by $\mathbb{P}^{n}$. Interpretations are such that
$\forall a\in\mathsf{A}(\left\llbracket a\right\rrbracket \in\mathbb{O})$
and $\forall F\in\mathsf{P}^{n}(\left\llbracket F\right\rrbracket \in\mathbb{P}^{n})$.
The interpretations of a set is the set of the interpretations of
its elements.
\end{defn}

To formulate a version of standard Tarskian model theory, we introduce
variable assignments, which are functions from variables to objects,
$\mathfrak{V}:\mathsf{X}\mapsto\mathbb{O}$. The variable assignment
$[x/\mathsf{o}]\mathfrak{v}$ is like the variable assignment $\mathfrak{v}$
except that it assigns the object $\mathsf{o}\in\mathbb{O}$ to the
variable $x\in\mathsf{X}$. We write $\left\llbracket \cdotp\right\rrbracket ^{\mathfrak{v}}$
for an interpretation function together with a variable assignment.
A model with a variable assignment satisfy a formula $\phi$ iff $\left\llbracket \phi\right\rrbracket ^{\mathfrak{v}}=1$.
And $\left\llbracket \phi\right\rrbracket ^{\mathfrak{v}}=0$ iff
$\left\llbracket \phi\right\rrbracket ^{\mathfrak{v}}\not=1$. The
semantic clauses are as follows:

\begin{itemize}
\item $\left\llbracket a\right\rrbracket ^{\mathfrak{v}}=\left\llbracket a\right\rrbracket $,
$\left\llbracket F\right\rrbracket ^{\mathfrak{v}}=\left\llbracket F\right\rrbracket $,
and $\left\llbracket x\right\rrbracket ^{\mathfrak{v}}=\mathfrak{v}(x)$.
\item $\forall F\nu_{1},...,\nu_{i}\in\mathsf{At}\left\llbracket F\nu_{1},...,\nu_{i}\right\rrbracket ^{\mathfrak{v}}=1$
iff $\left\langle \left\llbracket \nu_{1}\right\rrbracket ^{\mathfrak{v}},...,\left\llbracket \nu_{i}\right\rrbracket ^{\mathfrak{v}}\right\rangle \in\left\llbracket F\right\rrbracket ^{\mathfrak{v}}$.
\item $\left\llbracket \phi\wedge\psi\right\rrbracket ^{\mathfrak{v}}=1$
iff $\left\llbracket \phi\right\rrbracket ^{\mathfrak{v}}=1$ and
$\left\llbracket \psi\right\rrbracket ^{\mathfrak{v}}=1$.
\item $\left\llbracket \neg\phi\right\rrbracket ^{\mathfrak{v}}=1$ iff
$\left\llbracket \phi\right\rrbracket ^{\mathfrak{v}}=0$.
\item $\left\llbracket \forall x_{i}\phi\right\rrbracket ^{\mathfrak{v}}=1$
iff for every $\mathsf{o}\in\mathbb{O}$ $\left\llbracket \phi\right\rrbracket ^{[x_{i}/\mathsf{o}]\mathfrak{v}}=1$.
\end{itemize}

A sentence is true in a model if it is a closed formula that is satisfied
by the empty variable assignment together with that model. This is,
of course, a formulation of classical first-order logic (without identity
or functors, treating the conditional, disjunction, and the existential
quantifier as defined): $\Gamma\models_{FOL}\Delta$ iff, in every
Tarskian model in which every sentence in $\Gamma$ is true, at least
one of the sentences in $\Delta$ is true.

It will prove useful for our purposes to characterize truth-in-a-model
in a somewhat unfamiliar way. To do so, I introduce the notion of
states in Tarskian models.\footnote{One can think of this as a way to bring our formulation of Tarskian
model theory closer to truth-maker theory, which is useful because
implication-space semantics has known connections to truth-maker theory.
For reasons of space, I will not make these underlying connections
explicit here.}

\begin{defn}[States]
 Atomic states: For all $\mathsf{P}\in\mathbb{P}^{n}$ and $\overrightarrow{\mathsf{o}}\in\mathbb{O}^{n}$,
the pair $\left\langle \mathsf{P},\overrightarrow{\mathsf{o}}\right\rangle $
is an atomic state. Complex states: if $\Phi$ and $\Psi$ are states,
then so are $\sim\Phi$, and $\Phi\varcurlywedge\Psi$, and $\prod(\mathsf{o}_{i})\Phi$.
\end{defn}

Thus, the atomic states of a Tarskian model are the bearers of the
implication space that is build on the same objects and properties.
I do not specify what kinds of objects complex states are. As we will
see below, the only thing that will matter about such states is that
they can be partitioned in a certain way, and for that it does not
matter what kinds of objects they are.

We can substitute objects for each other in states. For atomic states,
the substitution of objects is the same as for bearers. For complex
states, substitution ``sees through'' the composition of the states,
in the following sense.

\begin{defn}[{Substitution of objects in states, $[\cdotp/\cdotp]\cdotp$}]
 For any atomic state $\mathsf{P}\overrightarrow{\mathsf{o}}$ and
objects $\mathtt{o}_{j}$ and $\mathtt{o}_{k}$, $[\mathtt{o}_{j}/\mathtt{o}_{k}]\mathsf{P}\overrightarrow{\mathsf{o}}$
is the state like $\mathsf{P}\overrightarrow{\mathsf{o}}$ except
that $\mathtt{o}_{k}$ is replaced by $\mathtt{o}_{j}$ everywhere
in $\overrightarrow{\mathsf{o}}$. 

If $\Phi$ is a complex state, then (i) $[\mathtt{o}_{j}/\mathtt{o}_{k}](\sim\Phi)=\sim[\mathtt{o}_{j}/\mathtt{o}_{k}]\Phi$,
and (ii) $[\mathtt{o}_{j}/\mathtt{o}_{k}](\Phi\varcurlywedge\Psi)=[\mathtt{o}_{j}/\mathtt{o}_{k}]\Phi\varcurlywedge[\mathtt{o}_{j}/\mathtt{o}_{k}]\Psi$,
and (iii) $[\mathtt{o}_{j}/\mathtt{o}_{k}](\prod(\mathsf{o}_{i})\Phi)=\prod(\mathsf{o}_{i})([\mathtt{o}_{j}/\mathtt{o}_{k}]\Phi)$
if $k\neq i$, and $[\mathtt{o}_{j}/\mathtt{o}_{k}](\prod(\mathsf{o}_{i})\Phi)=(\prod(\mathsf{o}_{i})\Phi)$
if $k=i$.

For a set of states $\mathsf{A}$, $[\mathtt{o}_{j}/\mathtt{o}_{k}]\mathsf{A}$
is the set of $[\mathtt{o}_{j}/\mathtt{o}_{k}]\Phi$ for all $\Phi\in\mathsf{A}$.
\end{defn}

We can now recursively define a partition of these states into two
sets, which I call $\hat{T}$ and $\hat{F}$. Intuitively, $\hat{T}$
is the set of sentences that are true in a model, and $\hat{F}$ is
the set of sentences that are false in a model.

\begin{defn}[Partition of states, $\left\langle \hat{T},\hat{F}\right\rangle $,
and its atomic basis, $\left\langle \hat{T_{0}},\hat{F_{0}}\right\rangle $ ]
 Atomic base partition: for atomic states, if $\overrightarrow{\mathsf{o}}\in\mathsf{P}$,
then $\left\langle \mathsf{P},\overrightarrow{\mathsf{o}}\right\rangle \in\hat{T_{0}}$.
If $\overrightarrow{\mathsf{o}}\not\in\mathsf{P}$, then $\left\langle \mathsf{P},\overrightarrow{\mathsf{o}}\right\rangle \in\hat{F_{0}}$.
For complex states, $\hat{T_{0}}\subset\hat{T}$, $\hat{F_{0}}\subset\hat{F}$,
and: (i) $\sim\Phi\in\hat{T}$ if $\Phi\in\hat{F}$; and $\sim\Phi\in\hat{F}$
if $\Phi\in\hat{T}$. (ii) $\Phi\varcurlywedge\text{\ensuremath{\Psi}}\in\hat{T}$
iff $\Phi\in\hat{T}$ and $\text{\ensuremath{\Psi}}\in\hat{T}$; and
$\Phi\varcurlywedge\text{\ensuremath{\Psi}}\ensuremath{\in}\hat{F}$
iff $\Phi\in\hat{F}$ or $\text{\ensuremath{\Psi}}\ensuremath{\in}\hat{F}$.
(iii) $\prod(\mathsf{o}_{i})\Phi\in\hat{T}$ if $\forall\mathsf{o}\in\mathbb{O}([\mathtt{o}/\mathtt{o}_{i}]\Phi\in\hat{T})$;
and $\prod(\mathsf{o}_{i})\Phi\in\hat{F}$ if $\exists\mathsf{o}\in\mathbb{O}([\mathtt{o}/\mathtt{o}_{i}]\Phi\in\hat{F})$.
\end{defn}

We can extend our interpretation function (without variable assignment)
to complex sentences as follows:
\begin{itemize}
\item $\forall Fa_{1},...,a_{i}\in\mathsf{At_{Sent}}\left\llbracket Fa_{1},...,a_{i}\right\rrbracket =\left\langle \left\llbracket F\right\rrbracket ,\left\langle \left\llbracket a_{1}\right\rrbracket ,...,\left\llbracket a_{i}\right\rrbracket \right\rangle \right\rangle $.
\item $\left\llbracket \phi\wedge\psi\right\rrbracket =\left\llbracket \phi\right\rrbracket \varcurlywedge\left\llbracket \psi\right\rrbracket $.
\item $\left\llbracket \neg\phi\right\rrbracket =\sim\left\llbracket \phi\right\rrbracket $.
\item $\left\llbracket \forall x_{i}[x_{i}/a]\phi\right\rrbracket =\prod(\left\llbracket a\right\rrbracket )\left\llbracket \phi\right\rrbracket $.
\end{itemize}
This allows us to characterize truth in a Tarskian model as follows:

\begin{lem}
$\chi$ is true in a Tarskian model, $\mathcal{M}_{T}$, iff $\left\llbracket \chi\right\rrbracket ^{\mathcal{M}_{T}}\in\hat{T}$
in that model.\label{lem:T-and-F-track-truth-in-T-models}
\end{lem}

\begin{proof}
It suffices to show that the empty variable assignment satisfies $\chi$
iff $\left\llbracket \chi\right\rrbracket \in\hat{T}$ and otherwise
$\left\llbracket \chi\right\rrbracket \in\hat{F}$. We argue by induction
on the complexity of $\chi$. The base case in which $\chi$ is an
atomic sentence is immediate.

If $\chi$ is $\neg\phi$, then $\sim\left\llbracket \phi\right\rrbracket \in\hat{T}$
iff $\left\llbracket \phi\right\rrbracket \in\hat{F}$. By our induction
hypothesis, this means that $\left\llbracket \phi\right\rrbracket ^{\emptyset}=0$
and so $\left\llbracket \neg\phi\right\rrbracket ^{\emptyset}=1$.
If $\chi$ is $\phi\wedge\psi$, then $\left\llbracket \phi\wedge\psi\right\rrbracket \in\hat{T}$
iff $\left\llbracket \phi\right\rrbracket \in\hat{T}$ and $\left\llbracket \psi\right\rrbracket \in\hat{T}$.
By our induction hypothesis, this happens iff $\left\llbracket \phi\right\rrbracket ^{\emptyset}=1$
and $\left\llbracket \psi\right\rrbracket ^{\emptyset}=1$, which
happens iff $\left\llbracket \phi\wedge\psi\right\rrbracket ^{\emptyset}=1$.
If $\chi$ is $\forall x_{i}\phi$, take a name $a_{j}$ that does
not occur in $\chi$. Then $[a_{j}/x_{i}]\phi=_{df.}\bar{\phi}$ is
like $\phi$ except that it is a sentence in which all unbound occurrences
of $x_{i}$ in $\phi$ have been replaced by the name $a_{j}$. We
must show that $\left\llbracket \forall x_{i}\phi\right\rrbracket ^{\emptyset}=1$
iff $\left\llbracket \forall x_{i}[x_{i}/a_{j}]\bar{\phi}\right\rrbracket \in\hat{T}$.
Now, $\left\llbracket \forall x_{i}\phi\right\rrbracket ^{\emptyset}=1$
iff for every $\mathsf{o}\in\mathbb{O}$ $\left\llbracket \phi\right\rrbracket ^{[x_{i}/\mathsf{o}]\mathfrak{v}}=1$.
By our induction hypothesis this happens iff $\forall\mathsf{o}\in\mathbb{O}([\mathtt{o}/\left\llbracket a_{j}\right\rrbracket ]\left\llbracket \bar{\phi}\right\rrbracket \in\hat{T})$,
i.e., iff $\prod(\left\llbracket a_{j}\right\rrbracket )\left\llbracket \bar{\phi}\right\rrbracket \in\hat{T}$,
which happens iff $\left\llbracket \forall x_{i}[x_{i}/a_{j}]\bar{\phi}\right\rrbracket \in\hat{T}$.
\end{proof}

This way of characterizing truth-in-a-model may seem unnecessarily
complicated. It will, however, prove useful in mapping Tarskian models
into implication-space models, which I will do in the next subsection.

\subsection{Mapping Tarskian Models into Implication-Space Models}

In this section, I introduce a mapping of Tarskian models into implication-space
models that preserves to which inferences a model is a counterexample.
In this way each Tarskian model is mapped to an implication-space
model that I will call its ``Ersatz Tarskian model.''

Recall that Tarskian models and implication-space models are both
build up from a set of objects and \emph{n}-ary properties (identified
with sets of \emph{n}-tupels of objects in Tarskian models) as well
as an interpretation function. We can exploit this commonality and
define Ersatz Tarskian models as follows:

\begin{defn}[Ersatz Tarskian Models]
 Given a Tarskian model, $\mathcal{M}_{T}$, its Ersatz Tarskian
implication-space model, $\mathcal{M}_{I}$, is the implication-space
model such that (i) the two models share the same objects $\mathbb{O}$
and properties $\mathbb{P}^{n}$ for all $n$, and (ii) $\mathbb{I}=\{\left\langle x,y\right\rangle \mid x\not\subseteq\hat{T_{0}}\thinspace\lor\thinspace y\not\subseteq\hat{F_{0}}\}$,
and (iii) their interpretations are parallel, in the sense that, for
all atomic sentences $\phi$, $\left\llbracket \phi\right\rrbracket ^{\mathcal{M}_{I}}=\mathcal{R}(\left\llbracket \phi\right\rrbracket ^{\mathcal{M}_{T}})$.

The following lemma will be useful. The intuitive idea behind the
lemma is that true sentences counterexample on the left and false
sentences counterexample on the right.
\end{defn}

\begin{lem}
Let $\mathcal{M}$ be the Ersatz Tarskian implication-space model
of the Tarskian model $\mathcal{M}_{T}$, then: if $\left\llbracket \phi\right\rrbracket ^{\mathcal{M}_{T}}\in\hat{T}$,
then $\left\llbracket \Gamma\cup\{\phi\}\right\rrbracket ^{\mathcal{M}}\not\dttstile{}{}\left\llbracket \Delta\right\rrbracket ^{\mathcal{M}}$
iff $\left\llbracket \Gamma\right\rrbracket ^{\mathcal{M}}\not\dttstile{}{}\left\llbracket \Delta\right\rrbracket ^{\mathcal{M}}$;
and if $\left\llbracket \phi\right\rrbracket ^{\mathcal{M}_{T}}\in\hat{F}$,
then $\left\llbracket \Gamma\right\rrbracket ^{\mathcal{M}}\not\dttstile{}{}\left\llbracket \Delta\cup\{\phi\}\right\rrbracket ^{\mathcal{M}}$
iff $\left\llbracket \Gamma\right\rrbracket ^{\mathcal{M}}\not\dttstile{}{}\left\llbracket \Delta\right\rrbracket ^{\mathcal{M}}$.\label{lem:T-counterexamples-on-left-and-F-on-right}
\end{lem}

\begin{proof}
Since $\mathbb{I}=\{\left\langle x,y\right\rangle \mid x\not\subseteq\hat{T_{0}}\thinspace\lor\thinspace y\not\subseteq\hat{F_{0}}\}$,
Ersatz Tarskian models are monotonic. So if $\left\llbracket \Gamma\cup\{\phi\}\right\rrbracket ^{\mathcal{M}}\not\dttstile{}{}\left\llbracket \Delta\right\rrbracket ^{\mathcal{M}}$
or $\left\llbracket \Gamma\right\rrbracket ^{\mathcal{M}}\not\dttstile{}{}\left\llbracket \Delta\cup\{\phi\}\right\rrbracket ^{\mathcal{M}}$,
then $\left\llbracket \Gamma\right\rrbracket ^{\mathcal{M}}\not\dttstile{}{}\left\llbracket \Delta\right\rrbracket ^{\mathcal{M}}$.
It hence suffices to show that if $\left\llbracket \Gamma\right\rrbracket ^{\mathcal{M}}\not\dttstile{}{}\left\llbracket \Delta\right\rrbracket ^{\mathcal{M}}$,
then (i) if $\left\llbracket \phi\right\rrbracket ^{\mathcal{M}_{T}}\in\hat{T}$,
then $\left\llbracket \Gamma\cup\{\phi\}\right\rrbracket ^{\mathcal{M}}\not\dttstile{}{}\left\llbracket \Delta\right\rrbracket ^{\mathcal{M}}$
and (ii) if $\left\llbracket \phi\right\rrbracket ^{\mathcal{M}_{T}}\in\hat{F}$,
then $\left\llbracket \Gamma\right\rrbracket ^{\mathcal{M}}\not\dttstile{}{}\left\llbracket \Delta\cup\{\phi\}\right\rrbracket ^{\mathcal{M}}$.
So let us suppose that $\left\llbracket \Gamma\right\rrbracket ^{\mathcal{M}}\not\dttstile{}{}\left\llbracket \Delta\right\rrbracket ^{\mathcal{M}}$,
i.e., $\bigcup(\stackrel[i=0]{n}{\sqcup}\mathtt{g}_{i}^{+}\thinspace\sqcup\thinspace\stackrel[j=0]{m}{\sqcup}\mathtt{d}_{j}^{-})\not\subseteq\mathbb{I}$
with $\left\llbracket \Gamma\right\rrbracket ^{\mathcal{M}}=\mathtt{G}=\{\mathtt{g}_{0},...,\mathtt{g}_{n}\}$
and $\left\llbracket \Delta\right\rrbracket ^{\mathcal{M}}=\mathtt{D}=\{\mathtt{d}_{0},...,\mathtt{d}_{m}\}$.
Since $\mathbb{I}=\{\left\langle x,y\right\rangle \mid x\not\subseteq\hat{T_{0}}\thinspace\lor\thinspace y\not\subseteq\hat{F_{0}}\}$,
we know that $\mathsf{A}_{i}\subseteq\hat{T_{0}}$ and $\mathsf{B}_{i}\subseteq\hat{F_{0}}$,
for all $\left\langle \mathsf{A}_{i},\mathsf{B}_{i}\right\rangle \in\bigcup(\stackrel[i=0]{n}{\sqcup}\mathtt{g}_{i}^{+}\thinspace\sqcup\thinspace\stackrel[j=0]{m}{\sqcup}\mathtt{d}_{j}^{-})$.

Recall that we can use the union of a role as our canonical representative
of the role. Hence, it suffices to show that for all $\left\langle \mathsf{A}_{i},\mathsf{B}_{i}\right\rangle \in\bigcup(\stackrel[i=0]{n}{\sqcup}\mathtt{g}_{i}^{+}\thinspace\sqcup\thinspace\stackrel[j=0]{m}{\sqcup}\mathtt{d}_{j}^{-})$,
(i) if $\left\llbracket \phi\right\rrbracket ^{\mathcal{M}_{T}}\in\hat{T}$
then $\{\left\langle \mathsf{A}_{i}\cup\mathsf{C},\mathsf{B}_{i}\cup\mathsf{D}\right\rangle \mid\left\langle \mathsf{C},\mathsf{D}\right\rangle \in\bigcup\left\llbracket \phi\right\rrbracket ^{\mathcal{M}_{I}+}\}\not\subseteq\mathbb{I}$
and (ii) if $\left\llbracket \phi\right\rrbracket ^{\mathcal{M}_{T}}\in\hat{F}$
then $\{\left\langle \mathsf{A}_{i}\cup\mathsf{C},\mathsf{B}_{i}\cup\mathsf{D}\right\rangle \mid\left\langle \mathsf{C},\mathsf{D}\right\rangle \in\bigcup\left\llbracket \phi\right\rrbracket ^{\mathcal{M}_{I}-}\}\not\subseteq\mathbb{I}$.
Since $\mathbb{I}=\{\left\langle x,y\right\rangle \mid x\not\subseteq\hat{T_{0}}\thinspace\lor\thinspace y\not\subseteq\hat{F_{0}}\}$,
this means that (i), whenever $\left\llbracket \phi\right\rrbracket ^{\mathcal{M}_{T}}\in\hat{T}$,
there is some $\left\langle \mathsf{C},\mathsf{D}\right\rangle \in\bigcup\left\llbracket \phi\right\rrbracket ^{\mathcal{M}_{I}+}$
such that $\mathsf{C}\subseteq\hat{T_{0}}$ and $\mathsf{D}\subseteq\hat{F_{0}}$;
and (ii), whenever $\left\llbracket \phi\right\rrbracket ^{\mathcal{M}_{T}}\in\hat{F}$,
there is some $\left\langle \mathsf{C},\mathsf{D}\right\rangle \in\bigcup\left\llbracket \phi\right\rrbracket ^{\mathcal{M}_{I}-}$
such that $\mathsf{C}\subseteq\hat{T_{0}}$ and $\mathsf{D}\subseteq\hat{F_{0}}$.
We now argue for this by induction on the complexity of $\left\llbracket \phi\right\rrbracket ^{\mathcal{M}_{T}}$.
Note that this is not the complexity of a sentence but the complexity
of a state; note also, however, that the complexity of a sentence
is equal to the complexity of the state that interprets it, according
to our new way of looking at truth in a Tarskian model.

\emph{Base case}: We know that, for all atomic states $\left\llbracket \phi\right\rrbracket ^{\mathcal{M}_{T}}$
and, hence, atomic sentences $\phi$ that $\left\llbracket \phi\right\rrbracket ^{\mathcal{M}_{I}}=\mathcal{R}(\left\llbracket \phi\right\rrbracket ^{\mathcal{M}_{T}})$
so that $\left\llbracket \phi\right\rrbracket ^{\mathcal{M}_{I}+}=\mathcal{R}\left\langle \{\left\llbracket \phi\right\rrbracket ^{\mathcal{M}_{T}}\},\emptyset\right\rangle $
and $\left\llbracket \phi\right\rrbracket ^{\mathcal{M}_{I}-}=\mathcal{R}\left\langle \emptyset,\{\left\llbracket \phi\right\rrbracket ^{\mathcal{M}_{T}}\}\right\rangle $.
Since $\left\langle \{\left\llbracket \phi\right\rrbracket ^{\mathcal{M}_{T}}\},\emptyset\right\rangle \in\mathcal{R}\left\langle \{\left\llbracket \phi\right\rrbracket ^{\mathcal{M}_{T}}\},\emptyset\right\rangle $,
we know that if $\left\llbracket \phi\right\rrbracket ^{\mathcal{M}_{T}}\in\hat{T}$,
then $\exists\left\langle \mathsf{C},\mathsf{D}\right\rangle \in\bigcup\left\llbracket \phi\right\rrbracket ^{\mathcal{M}_{I}+}$
such that $\mathsf{C}\subseteq\hat{T_{0}}$ and $\mathsf{D}\subseteq\hat{F_{0}}$.
After all, $\{\left\llbracket \phi\right\rrbracket ^{\mathcal{M}_{T}}\}\subseteq\hat{T_{0}}$
and $\emptyset\subseteq\hat{F_{0}}$. Similarly, if $\left\llbracket \phi\right\rrbracket ^{\mathcal{M}_{T}}\in\hat{F}$,
then $\exists\left\langle \mathsf{C},\mathsf{D}\right\rangle \in\bigcup\left\llbracket \phi\right\rrbracket ^{\mathcal{M}_{I}-}$
such that $\mathsf{C}\subseteq\hat{T_{0}}$ and $\mathsf{D}\subseteq\hat{F_{0}}$.
After all, $\emptyset\subseteq\hat{T_{0}}$ and $\{\left\llbracket \phi\right\rrbracket ^{\mathcal{M}_{T}}\}\subseteq\hat{F_{0}}$.

\emph{Induction step}: Our state $\left\llbracket \phi\right\rrbracket ^{\mathcal{M}_{T}}$
can be of the form $\sim\left\llbracket \eta\right\rrbracket ^{\mathcal{M}_{T}}$,
$\left\llbracket \eta\right\rrbracket ^{\mathcal{M}_{T}}\varcurlywedge\left\llbracket \theta\right\rrbracket ^{\mathcal{M}_{T}}$,
or $\prod(\mathsf{o})\left\llbracket \eta\right\rrbracket ^{\mathcal{M}_{T}}$.

State-Negation: If $\left\llbracket \phi\right\rrbracket ^{\mathcal{M}_{T}}=\thinspace\sim\left\llbracket \eta\right\rrbracket ^{\mathcal{M}_{T}}$,
there are sentences $\phi=\neg\eta$. Then $\left\llbracket \phi\right\rrbracket ^{\mathcal{M}_{T}}\in\hat{T}$
iff $\left\llbracket \eta\right\rrbracket ^{\mathcal{M}_{T}}\in\hat{F}$.
By our induction hypothesis, $\exists\left\langle \mathsf{C},\mathsf{D}\right\rangle \in\bigcup\left\llbracket \eta\right\rrbracket ^{\mathcal{M}_{I}+}$
such that $\mathsf{C}\subseteq\hat{T_{0}}$ and $\mathsf{D}\subseteq\hat{F_{0}}$
whenever $\left\llbracket \eta\right\rrbracket ^{\mathcal{M}_{T}}\in\hat{T}$;
and $\exists\left\langle \mathsf{C},\mathsf{D}\right\rangle \in\bigcup\left\llbracket \eta\right\rrbracket ^{\mathcal{M}_{I}-}$
such that $\mathsf{C}\subseteq\hat{T_{0}}$ and $\mathsf{D}\subseteq\hat{F_{0}}$
whenever $\left\llbracket \eta\right\rrbracket ^{\mathcal{M}_{T}}\in\hat{F}$.
But $\left\llbracket \eta\right\rrbracket ^{\mathcal{M}_{I}+}=\left\llbracket \phi\right\rrbracket ^{\mathcal{M}_{I}-}$
and $\left\llbracket \eta\right\rrbracket ^{\mathcal{M}_{I}-}=\left\llbracket \phi\right\rrbracket ^{\mathcal{M}_{I}+}$.
So, $\exists\left\langle \mathsf{C},\mathsf{D}\right\rangle \in\bigcup\left\llbracket \phi\right\rrbracket ^{\mathcal{M}_{I}-}$such
that $\mathsf{C}\subseteq\hat{T_{0}}$ and $\mathsf{D}\subseteq\hat{F_{0}}$
whenever $\left\llbracket \phi\right\rrbracket ^{\mathcal{M}_{T}}\in\hat{F}$;
and $\exists\left\langle \mathsf{C},\mathsf{D}\right\rangle \in\bigcup\left\llbracket \phi\right\rrbracket ^{\mathcal{M}_{I}+}$
such that $\mathsf{C}\subseteq\hat{T_{0}}$ and $\mathsf{D}\subseteq\hat{F_{0}}$
whenever $\left\llbracket \phi\right\rrbracket ^{\mathcal{M}_{T}}\in\hat{T}$.

State-Conjunction: If $\left\llbracket \phi\right\rrbracket ^{\mathcal{M}_{T}}=\thinspace\left\llbracket \eta\right\rrbracket ^{\mathcal{M}_{T}}\varcurlywedge\left\llbracket \theta\right\rrbracket ^{\mathcal{M}_{T}}$,
there are sentences $\phi=\eta\wedge\theta$. Then $\left\llbracket \phi\right\rrbracket ^{\mathcal{M}_{T}}\in\hat{T}$
iff $\left\llbracket \eta\right\rrbracket ^{\mathcal{M}_{T}}\in\hat{T}$
and $\left\llbracket \theta\right\rrbracket ^{\mathcal{M}_{T}}\in\hat{T}$,
and $\left\llbracket \phi\right\rrbracket ^{\mathcal{M}_{T}}\in\hat{F}$
otherwise. By our induction hypothesis, $\exists\left\langle \mathsf{C},\mathsf{D}\right\rangle \in\bigcup\left\llbracket \eta\right\rrbracket ^{\mathcal{M}_{I}+}$
such that $\mathsf{C}\subseteq\hat{T_{0}}$ and $\mathsf{D}\subseteq\hat{F_{0}}$
whenever $\left\llbracket \eta\right\rrbracket ^{\mathcal{M}_{T}}\in\hat{T}$;
and (ii) $\exists\left\langle \mathsf{C},\mathsf{D}\right\rangle \in\bigcup\left\llbracket \eta\right\rrbracket ^{\mathcal{M}_{I}-}$
such that $\mathsf{C}\subseteq\hat{T_{0}}$ and $\mathsf{D}\subseteq\hat{F_{0}}$
whenever $\left\llbracket \eta\right\rrbracket ^{\mathcal{M}_{T}}\in\hat{F}$;
and the corresponding claim holds for $\theta$. Now, $\left\llbracket \eta\wedge\theta\right\rrbracket ^{\mathcal{M}_{I}+}=\left\llbracket \eta\right\rrbracket ^{\mathcal{M}_{I}+}\sqcup\left\llbracket \theta\right\rrbracket ^{\mathcal{M}_{I}+}$
and $\left\llbracket \eta\wedge\theta\right\rrbracket ^{\mathcal{M}_{I}-}=$
$\left\llbracket \eta\right\rrbracket ^{\mathcal{M}_{I}-}\sqcap\left\llbracket \theta\right\rrbracket ^{\mathcal{M}_{I}-}\sqcap(\left\llbracket \eta\right\rrbracket ^{\mathcal{M}_{I}-}\sqcup\left\llbracket \theta\right\rrbracket ^{\mathcal{M}_{I}-})$.
Note that, by the definition of symjunction, $\bigcup\left\llbracket \eta\right\rrbracket ^{\mathcal{M}_{I}-}\subseteq\bigcup\left\llbracket \eta\wedge\theta\right\rrbracket ^{\mathcal{M}_{I}-}$.
Case (i): Suppose that $\left\llbracket \phi\right\rrbracket ^{\mathcal{M}_{T}}\in\hat{T}$.
Then $\left\llbracket \eta\right\rrbracket ^{\mathcal{M}_{T}}\in\hat{T}$
and $\left\llbracket \theta\right\rrbracket ^{\mathcal{M}_{T}}\in\hat{T}$.
Hence, there is a pair $\left\langle \mathsf{C},\mathsf{D}\right\rangle \in\bigcup\left\llbracket \eta\right\rrbracket ^{\mathcal{M}_{I}+}$
such that $\mathsf{C}\subseteq\hat{T_{0}}$ and $\mathsf{D}\subseteq\hat{F_{0}}$
and there is a pair $\left\langle \mathsf{E},\mathsf{F}\right\rangle \in\bigcup\left\llbracket \theta\right\rrbracket ^{\mathcal{M}_{I}+}$
such that $\mathsf{E}\subseteq\hat{T_{0}}$ and $\mathsf{F}\subseteq\hat{F_{0}}$.
It follows that there is a pair $\left\langle \mathsf{C}\cup\mathsf{E},\mathsf{D}\cup\mathsf{F}\right\rangle \in\bigcup(\left\llbracket \eta\right\rrbracket ^{\mathcal{M}_{I}+}\sqcup\left\llbracket \theta\right\rrbracket ^{\mathcal{M}_{I}+})=\bigcup\left\llbracket \eta\wedge\theta\right\rrbracket ^{\mathcal{M}_{I}+}$
such $\mathsf{C}\cup\mathsf{E}\subseteq\hat{T_{0}}$ and $\mathsf{D}\cup\mathsf{F}\subseteq\hat{F_{0}}$.
Case (ii): Suppose $\left\llbracket \phi\right\rrbracket ^{\mathcal{M}_{T}}\in\hat{F}$.
Then either $\left\llbracket \eta\right\rrbracket ^{\mathcal{M}_{T}}\in\hat{F}$
or $\left\llbracket \theta\right\rrbracket ^{\mathcal{M}_{T}}\in\hat{F}$.
Suppose without loss of generality that $\left\llbracket \eta\right\rrbracket ^{\mathcal{M}_{T}}\in\hat{F}$.
By our induction hypothesis, $\exists\left\langle \mathsf{C},\mathsf{D}\right\rangle \in\bigcup\left\llbracket \eta\right\rrbracket ^{\mathcal{M}_{I}-}$
such that $\mathsf{C}\subseteq\hat{T_{0}}$ and $\mathsf{D}\subseteq\hat{F_{0}}$.
Since $\bigcup\left\llbracket \eta\right\rrbracket ^{\mathcal{M}_{I}-}\subseteq\bigcup\left\llbracket \eta\wedge\theta\right\rrbracket ^{\mathcal{M}_{I}-}$,
it follows that $\exists\left\langle \mathsf{C},\mathsf{D}\right\rangle \in\bigcup\left\llbracket \eta\wedge\theta\right\rrbracket ^{\mathcal{M}_{I}-}$
such that $\mathsf{C}\subseteq\hat{T_{0}}$ and $\mathsf{D}\subseteq\hat{F_{0}}$.

State-Universal-Generalization: If $\left\llbracket \phi\right\rrbracket ^{\mathcal{M}_{T}}=\thinspace\prod(\mathsf{o})\left\llbracket \eta\right\rrbracket ^{\mathcal{M}_{T}}$,
there are sentences $\phi$ and $\eta$ such that $\phi=\forall x_{i}[x_{i}/a]\eta$.
Then $\left\llbracket \phi\right\rrbracket ^{\mathcal{M}_{T}}\in\hat{T}$
iff $\forall\mathsf{o}\in\mathbb{O}([\mathsf{o}/\left\llbracket a\right\rrbracket ^{\mathcal{M}_{T}}]\left\llbracket \eta\right\rrbracket ^{\mathcal{M}_{T}}\in\hat{T})$;
and $\left\llbracket \phi\right\rrbracket ^{\mathcal{M}_{T}}\in\hat{F}$
iff $\exists\mathsf{o}\in\mathbb{O}([\mathsf{o}/\left\llbracket a\right\rrbracket ^{\mathcal{M}_{T}}]\left\llbracket \eta\right\rrbracket ^{\mathcal{M}_{T}}\in\hat{F})$.
By our induction hypothesis, for all $[\mathsf{o}/\left\llbracket a\right\rrbracket ^{\mathcal{M}_{T}}]\left\llbracket \eta\right\rrbracket ^{\mathcal{M}_{T}}$,
we have (i) $\exists\left\langle \mathsf{C},\mathsf{D}\right\rangle \in\bigcup[\mathsf{o}/\left\llbracket a\right\rrbracket ^{\mathcal{M}_{T}}]\left\llbracket \eta\right\rrbracket ^{\mathcal{M}_{T}+}$
such that $\mathsf{C}\subseteq\hat{T_{0}}$ and $\mathsf{D}\subseteq\hat{F_{0}}$
whenever $[\mathsf{o}/\left\llbracket a\right\rrbracket ^{\mathcal{M}_{T}}]\left\llbracket \eta\right\rrbracket ^{\mathcal{M}_{T}}\in\hat{T}$;
and (ii) $\exists\left\langle \mathsf{C},\mathsf{D}\right\rangle \in\bigcup[\mathsf{o}/\left\llbracket a\right\rrbracket ^{\mathcal{M}_{T}}]\left\llbracket \eta\right\rrbracket ^{\mathcal{M}_{T}-}$
such that $\mathsf{C}\subseteq\hat{T_{0}}$ and $\mathsf{D}\subseteq\hat{F_{0}}$
whenever $[\mathsf{o}/\left\llbracket a\right\rrbracket ^{\mathcal{M}_{T}}]\left\llbracket \eta\right\rrbracket ^{\mathcal{M}_{T}}\in\hat{F}$.
Moreover, $\left\llbracket \forall x_{i}[x_{i}/a]\eta\right\rrbracket ^{\mathcal{M}_{I}+}$
is $\bigsqcup\{[\mathsf{o}/\left\llbracket a\right\rrbracket ^{\mathcal{M}_{T}}]\left\llbracket \eta\right\rrbracket ^{\mathcal{M}_{I}+}\mid\mathsf{o}\in\mathbb{O}\}$;
and $\left\llbracket \forall x_{i}[x_{i}/a]\eta\right\rrbracket ^{\mathcal{M}_{I}-}$
is $\APLdown\{[\mathsf{o}/\left\llbracket a\right\rrbracket ^{\mathcal{M}_{T}}]\left\llbracket \eta\right\rrbracket ^{\mathcal{M}_{I}-}\mid\mathsf{o}\in\mathbb{O}\}$.

Case (i): Suppose that $\left\llbracket \phi\right\rrbracket ^{\mathcal{M}_{T}}\in\hat{T}$.
Then $\forall\mathsf{o}\in\mathbb{O}([\mathsf{o}/\left\llbracket a\right\rrbracket ^{\mathcal{M}_{T}}]\left\llbracket \eta\right\rrbracket ^{\mathcal{M}_{T}}\in\hat{T})$.
So, for each of the corresponding roles, $\exists\left\langle \mathsf{C},\mathsf{D}\right\rangle \in\bigcup[\mathsf{o}/\left\llbracket a\right\rrbracket ^{\mathcal{M}_{T}}]\left\llbracket \eta\right\rrbracket ^{\mathcal{M}_{T}+}$
such that $\mathsf{C}\subseteq\hat{T_{0}}$ and $\mathsf{D}\subseteq\hat{F_{0}}$.
Call this the counterexampleing pair, $\left\langle \mathsf{C}_{\mathsf{o}},\mathsf{D}_{\mathsf{o}}\right\rangle $,
for $[\mathsf{o}/\left\llbracket a\right\rrbracket ^{\mathcal{M}_{T}}]\left\llbracket \eta\right\rrbracket ^{\mathcal{M}_{T}+}$.
Now $\bigcup\bigsqcup\{[\mathsf{o}/\left\llbracket a\right\rrbracket ^{\mathcal{M}_{T}}]\left\llbracket \eta\right\rrbracket ^{\mathcal{M}_{I}+}\mid\mathsf{o}\in\mathbb{O}\}$
contains a pair $\left\langle \mathsf{C}*,\mathsf{D}*\right\rangle =\left\langle \bigcup\{\mathsf{C}_{\mathsf{o}}\mid\mathsf{o}\in\mathbb{O}\},\bigcup\{\mathsf{D}_{\mathsf{o}}\mid\mathsf{o}\in\mathbb{O}\}\right\rangle $,
i.e., a pair such that $\mathsf{C}*$ is the union of all the first
elements of a counterexampling pair for $[\mathsf{o}/\left\llbracket a\right\rrbracket ^{\mathcal{M}_{T}}]\left\llbracket \eta\right\rrbracket ^{\mathcal{M}_{T}+}$
for each object and $\mathsf{D}*$ is the union of all the second
elements of a counterexampling pair for $[\mathsf{o}/\left\llbracket a\right\rrbracket ^{\mathcal{M}_{T}}]\left\llbracket \eta\right\rrbracket ^{\mathcal{M}_{T}+}$
for each object. It follows $\mathsf{C}*\subseteq\hat{T_{0}}$ and
$\mathsf{D}*\subseteq\hat{F_{0}}$. Therefore, $\exists\left\langle \mathsf{C},\mathsf{D}\right\rangle \in\bigcup\left\llbracket \phi\right\rrbracket ^{\mathcal{M}_{I}+}$
such that $\mathsf{C}\subseteq\hat{T_{0}}$ and $\mathsf{D}\subseteq\hat{F_{0}}$.

Case (ii): Suppose $\left\llbracket \phi\right\rrbracket ^{\mathcal{M}_{T}}\in\hat{F}$.
Then $\exists\mathsf{o}\in\mathbb{O}([\mathsf{o}/\left\llbracket a\right\rrbracket ^{\mathcal{M}_{T}}]\left\llbracket \eta\right\rrbracket ^{\mathcal{M}_{T}}\in\hat{F})$.
Suppose without loss of generality that $[\mathsf{o}/\left\llbracket a\right\rrbracket ^{\mathcal{M}_{T}}]\left\llbracket \eta\right\rrbracket ^{\mathcal{M}_{T}}\in\hat{F}$.
By our induction hypothesis, $\exists\left\langle \mathsf{C},\mathsf{D}\right\rangle \in\bigcup[\mathsf{o}/\left\llbracket a\right\rrbracket ^{\mathcal{M}_{T}}]\left\llbracket \eta\right\rrbracket ^{\mathcal{M}_{I}-}$
such that $\mathsf{C}\subseteq\hat{T_{0}}$ and $\mathsf{D}\subseteq\hat{F_{0}}$.
Since $\bigcup[\mathsf{o}/\left\llbracket a\right\rrbracket ^{\mathcal{M}_{T}}]\left\llbracket \eta\right\rrbracket ^{\mathcal{M}_{I}-}\subseteq\bigcup\left\llbracket \forall x_{i}[x_{i}/a]\eta\right\rrbracket ^{\mathcal{M}_{I}-}$,
it follows that $\exists\left\langle \mathsf{C},\mathsf{D}\right\rangle \in\bigcup\left\llbracket \forall x_{i}[x_{i}/a]\eta\right\rrbracket ^{\mathcal{M}_{I}-}$
such that $\mathsf{C}\subseteq\hat{T_{0}}$ and $\mathsf{D}\subseteq\hat{F_{0}}$.
\end{proof}

\begin{prop}
Each Tarskian model, $\mathcal{M}_{T}$, and its Ersatz Tarskian implication-space
model, $\mathcal{M}_{I}$, are counterexamples to the same inferences.
\end{prop}

\begin{proof}
$\mathcal{M}_{T}$ is a counterexample to $\Gamma\vdash\Delta$ if
every sentence in $\Gamma$ is true in $\mathcal{M}_{T}$ and every
sentence in $\Delta$ is not true in $\mathcal{M}_{T}$. By induction
on the number of sentences in $\Gamma\vdash\Delta$, it follows from
Lemmas~\ref{lem:T-and-F-track-truth-in-T-models} and \ref{lem:T-counterexamples-on-left-and-F-on-right}
that both, $\mathcal{M}_{T}$ and $\mathcal{M}_{I}$, counterexample
an inference $\Gamma\vdash\Delta$ iff $\left\llbracket \Gamma\right\rrbracket ^{\mathcal{M}_{T}}\in\hat{T}$
and $\left\llbracket \Delta\right\rrbracket ^{\mathcal{M}_{T}}\in\hat{F}$.
\end{proof}

This means that mapping each standard Tarskian model to its Ersatz
Tarskian model preserves the set of inferences that are counterexampled
by the set of models. So, if we take the Ersatz Tarskian models of
all the standard Tarskian models, then we arrive at a set of implication-space
models that include a counterexample to all and only the inferences
that are invalid in classical first-order logic.

\begin{thm}
An inference is valid in classical first-order logic iff the implication
holds in every Ersatz Tarskian implication-space model.\label{thm:FOL-maps-to-Ersatz-models}
\end{thm}

\begin{proof}
The set of Ersatz Tarskian implication-space models are all and only
models such that there is a Tarskian model that counterexamples exactly
the same inferences.
\end{proof}

We can thus recover classical first-order logic in implication-space
semantics by letting the models over which we define implication be
the Ersatz Tarskian models. One may wonder, however, whether we can
characterize this set of models without appealing to Tarskian model
theory.

As it turns out, the Ersatz Tarskian models are counterexamples to
all and only the inferences that are also counterexampled by the models
that obey Gentzen's structural rules in the implication relation of
their frames, which we may call the structural models. We can define
these models thus:

\begin{defn}[Structural models]
 An implication-space model is structural iff
\begin{itemize}
\item (CO, containment / reflexivity {[}with contexts{]}) for all $\mathsf{G},\mathsf{D}\subseteq\mathbb{B}$,
if $\mathsf{G}\cap\mathsf{D}\neq\emptyset$, then $\left\langle \mathsf{G},\mathsf{D}\right\rangle \in\mathbb{I}$
and\vspace{-6pt}
\item (MO, monotonicity) for all $\mathsf{G},\mathsf{D},\mathsf{E},\mathsf{F}\subseteq\mathbb{B}$,
if $\left\langle \mathsf{G},\mathsf{D}\right\rangle \in\mathbb{I}$,
then $\left\langle \mathsf{G}\cup\mathsf{E},\mathsf{D}\cup\mathsf{F}\right\rangle \in\mathbb{I}$
and\vspace{-6pt}
\item (CUT, transitivity) for all $\mathsf{G},\mathsf{D}\subseteq\mathbb{B}$
and $\mathsf{a}\in\mathbb{B}$, if $\left\langle \mathsf{G},\mathsf{D}\right\rangle \not\in\mathbb{I}$,
then $\left\langle \mathsf{G}\cup\{\mathsf{a}\},\mathsf{D}\right\rangle \not\in\mathbb{I}$
or $\left\langle \mathsf{G},\mathsf{D\cup\{\mathsf{a}\}}\right\rangle \not\in\mathbb{I}$.
\end{itemize}
\end{defn}

I will now show that if $m$ is the set of all structural implication
space models, then $\Gamma\sttstile{}{m}\Delta$ iff $\Gamma\models_{FOL}\Delta$.
It suffices to show that (a) each structural model is equivalent to
a set of Ersatz Tarskian models, in its power to provide counterexamples,
and (b) all Ersatz Tarskian models are structural models. The second
part is easy to show.

\begin{lem}
All Ersatz Tarskian implication-space models are structural models.\label{lem:Ersatz-models-are-structural}
\end{lem}

\begin{proof}
By definition, Ersatz Tarskian models are such that $\mathbb{I}=\{\left\langle x,y\right\rangle \mid x\not\subseteq\hat{T}\thinspace\lor\thinspace y\not\subseteq\hat{F}\}$.
Equivalently, $\left\langle x,y\right\rangle \not\in\mathbb{I}$ iff
$x\subseteq\hat{T}$ and $y\subseteq\hat{F}$. In every Tarskian model
$\hat{T}$ and $\hat{F}$ partition the states of affairs.

CO: If $\mathsf{G}\cap\mathsf{D}\neq\emptyset$, then it cannot be
that $\mathsf{G}\subseteq\hat{T}$ and $\mathsf{D}\subseteq\hat{F}$
because $\hat{T}$ and $\hat{F}$ are disjoint. So $\left\langle \mathsf{G},\mathsf{D}\right\rangle \in\mathbb{I}$.

MO: If $\left\langle \mathsf{G},\mathsf{D}\right\rangle \in\mathbb{I}$,
then $\mathsf{G}\not\subseteq\hat{T}$ or $\mathsf{D}\not\subseteq\hat{F}$.
But then $\mathsf{G}\cup\mathsf{E}\not\subseteq\hat{T}$ or $\mathsf{D}\cup\mathsf{F}\not\subseteq\hat{F}$.
So $\left\langle \mathsf{G}\cup\mathsf{E},\mathsf{D}\cup\mathsf{F}\right\rangle \in\mathbb{I}$.

CUT: If $\left\langle \mathsf{G},\mathsf{D}\right\rangle \not\in\mathbb{I}$,
then $\mathsf{G}\subseteq\hat{T}$ and $\mathsf{D}\subseteq\hat{F}$.
Since $\hat{T}$ and $\hat{F}$ partition the states of affairs and
$\mathsf{a}\in\mathbb{B}$ (i.e., $\mathsf{a}$ is a state of affairs),
either $\mathsf{a}\in\hat{T}$ or $\mathsf{a}\in\hat{F}$. Therefore,
either $\mathsf{G}\cup\{\mathsf{a}\}\subseteq\hat{T}$ or $\mathsf{D}\cup\{\mathsf{a}\}\subseteq\hat{F}$.
Hence, $\left\langle \mathsf{G}\cup\{\mathsf{a}\},\mathsf{D}\right\rangle \not\in\mathbb{I}$
or $\left\langle \mathsf{G},\mathsf{D\cup\{\mathsf{a}\}}\right\rangle \not\in\mathbb{I}$.
\end{proof}

Establishing the connection in the other direction requires more work
because some structural models correspond not to one but to a whole
set of Ersatz Tarskian models. For readers at home in proof theory,
to understand the idea behind the following proof, it may be helpful
to realize that when we construct a Tarskian countermodel to an invalid
inference based on a root-first proof search in a sequent calculus,
we find an invalid leaf, containing only atomic sentences, of the
``would-be'' proof-tree and we ensure that all the atomic sentences
on the left are true in the model we are constructing, while all the
atomic sentences on the right are false. Cut then ensures that we
can extend this truth-value assignment to cover all atomic sentences.
In this way, a Tarskian model corresponds to a leafs of a ``would-be''
proof-tree of an underivable sequent. In implication-space semantics,
we do not have to pick one of these leafs; rather, we can construct
a model that is a counterexample to all such leafs simultaneously.
The tree of partitions in the construction below, in effect, mirrors
the structure of the branches of a ``would-be'' proof-tree with
invalid leafs.

\begin{lem}
For every structural implication-space model, there is a set of Ersatz
Tarskian models such that the implication-space model is a counterexample
to all and only the inferences to which some of the Ersatz Tarskian
model in that set is a counterexample.\label{lem:Structural-models-are-sets-of-ersatz-models}
\end{lem}

\begin{proof}
We must show that, for every structural implication-space model, $\mathcal{M}$,
there is a set of partitions of $\mathbb{B}$ into two sets, $\hat{T}_{k}$
and $\hat{F}_{k}$, such that $\left\langle x,y\right\rangle \not\in\mathbb{I}$
iff $x\subseteq\hat{T_{k}}$ and $y\subseteq\hat{F_{k}}$ for some
Ersatz Tarskian implication-space model $\mathcal{M}_{k}$. (If we
want to include a trivial model in which $\left\langle \emptyset,\emptyset\right\rangle \in\mathbb{I}^{\mathcal{M}}$,
then by MO every candidate implication is good. So we can put all
bearers into$\hat{T}$ (or all into $\hat{F}$) and that will give
us the desired partition.)

Construction of partitions: We order our bearers (thus using the axiom
of choice) and let the collection of all such orderings be $\mathfrak{U}$.
For each such ordering, $\mathfrak{U}_{i}$, we construct a set of
partitions, $\mathfrak{P}_{i}$, of $\mathbb{B}$ of the desired kind
and these partitions form a tree in the following way. Since we know
that $\left\langle \emptyset,\emptyset\right\rangle \not\in\mathbb{I}^{\mathcal{M}}$,
this is our starting point $\mathfrak{P}_{i}=\{\left\langle \emptyset,\emptyset\right\rangle \}$
at stage 0. We run through all of the bearers (always taking the least
element of $\mathfrak{U}_{i}$ minus the bearers that we have already
added to our partitions). At stage $n$, for all $\left\langle \hat{T_{j}},\hat{F_{j}}\right\rangle \in\mathfrak{P}_{i}$
we have $\left\langle \hat{T_{j}},\hat{F_{j}}\right\rangle \not\in\mathbb{\mathbb{I}^{\mathcal{M}}}$.
For the next bearer $\mathsf{b}_{n+1}$, in $\mathfrak{U}_{i}$, we
know from CUT that, for all $\left\langle \hat{T_{j}},\hat{F_{j}}\right\rangle \in\mathfrak{P}_{i}$,
either $\left\langle \{\mathsf{b}_{n+1}\}\cup\hat{T_{j}},\hat{F_{j}}\right\rangle \not\in\mathbb{I}^{\mathcal{M}}$
or $\left\langle \hat{T_{j}},\{\mathsf{b}_{n+1}\}\cup\hat{F_{j}}\right\rangle \not\in\mathbb{I}^{\mathcal{M}}$
or both. In the first case, $\left\langle \hat{T_{j}},\hat{F_{j}}\right\rangle $
at stage $n+1$ is $\left\langle \{\mathsf{b}_{n+1}\}\cup\hat{T_{j}},\hat{F_{j}}\right\rangle $;
in the second case, it is $\left\langle \hat{T_{j}},\{\mathsf{b}_{n+1}\}\cup\hat{F_{j}}\right\rangle $.
In the third case, $\left\langle \hat{T_{j}},\hat{F_{j}}\right\rangle $
splits into two pairs and $\mathfrak{P}_{i}$ at stage $n+1$ contains
$\left\langle \{\mathsf{b}_{n+1}\}\cup\hat{T_{j}},\hat{F_{j}}\right\rangle $
and also $\left\langle \hat{T_{j}},\{\mathsf{b}_{n+1}\}\cup\hat{F_{j}}\right\rangle $.
The result of this construction after running through all the bearers
in the ordering $\mathfrak{U}_{i}$ is the set of partitions $\mathfrak{P}_{i}$.
The union of all these sets of partitions for all orderings $\mathfrak{U}$
is $\mathfrak{P}$.

From partitions to Ersatz Tarskian models: For all $\left\langle \hat{T_{k}},\hat{F_{k}}\right\rangle \in\mathfrak{P}$,
since we added every bearer either to $\hat{T_{k}}$ or to $\hat{F_{k}}$
at some point, we know that our sets jointly exhaust $\mathbb{B}$.
And since our construction ensures that $\left\langle \hat{T_{k}},\hat{F_{k}}\right\rangle \not\in\mathbb{I}^{\mathcal{M}}$,
CO implies that $\hat{T_{k}}\cap\hat{F_{k}}=\emptyset$. So, we know
that every $\left\langle \hat{T_{k}},\hat{F_{k}}\right\rangle \in\mathfrak{P}$
is a partition of $\mathbb{B}$. The model that is like $\mathcal{M}$
except that $\mathbb{I}^{\mathcal{M}_{k}}=\{\left\langle x,y\right\rangle \mid x\not\subseteq\hat{T_{k}}\thinspace\lor\thinspace y\not\subseteq\hat{F_{k}}\}$
is an Ersatz Tarskian model such that $\hat{T_{k}}$ are the atomic
states that are true in the model and $\hat{F_{k}}$ are the atomic
states that are not true in the model. Hence, for every $\left\langle \hat{T_{k}},\hat{F_{k}}\right\rangle \in\mathfrak{P}$,
we have an Ersatz Tarskian model $\mathcal{M}_{k}$.

Notice that $\left\langle x,y\right\rangle \not\in\mathbb{I}^{\mathcal{M}}$
iff there is some $\left\langle \hat{T_{k}},\hat{F_{k}}\right\rangle \in\mathfrak{P}$
such that $x\subseteq\hat{T_{k}}$ and $y\subseteq\hat{F_{k}}$. After
all, if $\left\langle x,y\right\rangle \not\in\mathbb{I}^{\mathcal{M}}$,
then our construction will add all the elements of $x$ and $y$ to
some pair in $\mathfrak{P}$ at some point. Hence, if our model $\mathcal{M}$
is a counterexample to some implication $\left\langle x,y\right\rangle $,
then one of the Ersatz Tarskian models $\mathcal{M}_{k}$ is also
a counterexample to $\left\langle x,y\right\rangle $. It follows
that $\mathcal{M}$ is a counterexample to an inference iff there
is a model $\mathcal{M}_{k}$ in our construction that is a counterexample
to that inference.
\end{proof}

If we put these two lemmas together, we can see that the inferences
that hold in all structural models are exactly the inferences that
are valid in first-order logic.

\begin{thm}
An inference is valid in classical first-order logic iff there is
no structural implication-space model that is a counterexample to
the inference. In other words, with $m$ being the set of all structural
implication space models, $\Gamma\sttstile{}{m}\Delta$ iff $\Gamma\models_{FOL}\Delta$.\label{thm:FOL-is-structural-imp-space}
\end{thm}

\begin{proof}
By Theorem~\ref{thm:FOL-maps-to-Ersatz-models}, an inference is
valid in classical first-order logic iff there is no Ersatz Tarskian
implication-space model that is a counterexample. By Lemma~\ref{lem:Ersatz-models-are-structural},
if there is an Ersatz Tarskian model that is such a counterexample,
then there is a structural model that is a counterexample. Hence,
if the inference is classically invalid, then there is a structural
counterexample. For the other direction, by Lemma~\ref{lem:Structural-models-are-sets-of-ersatz-models},
if there is a structural model that is a counterexample to the inference,
then there is an Ersatz Tarskian model that is a counterexample. So,
if there is a structural counterexample, then the inference is classically
invalid.
\end{proof}

This allows us to characterize the implication-space models that yield
first-order classical logic without appealing to the models of Tarskian
model theory. The set of structural models will do the job. Since
these models provide counterexamples to all material inferences, we
must restrict ourselves to a subset of these structural models in
order to validate any material inferences. Moreover, we will have
to include some models that are not structural but in which monotonicity
fails in order to validate defeasible, material inferences. Thus,
the question arises how we can validate inferences that allow us to
meet desiderata like MOF and the others from above, while keeping
our consequence relation supraclassical.

A general answer to this question starts by noting that the consequence
relation of a set of implication-space models is the intersection
of the consequence relations of all these models, i.e., the singletons
of these models. Hence, the consequence relation of a set of models
is supraclassical iff the consequence relation of each of these models
is supraclassical. Besides all the classically valid inferences, a
model may validate further inferences; and with respect to these further
inferences monotonicity and transitivity can fail.

In Tarskian model theory, no analogous situation can arise because
there isn't any Tarskian model that is a counterexample to all classically
invalid inferences. In implication-space semantics, however, there
is such a model.

\subsection{One Model is Enough for FOL}

There is one implication-space model, for a given language, that counterexamples
all and only the inferences that are invalid in classical first-order
logic. I call it the canonical FOL model.

\begin{defn}[Canonical FOL implication-space model, $\mathcal{M}_{C}$]
 The canonical implication space model for a language, $\mathfrak{L}$,
is the implication-space model in which $\mathbb{O}$ are the natural
numbers assigned to names in their alphabetical order, $\mathbb{P}$
are sets of \emph{n}-tuples of natural numbers assigned to \emph{n}-ary
predicates in their alphabetical order (so that atomic sentences are
mapped to roles of individual bearers), and $\left\langle \mathsf{G},\mathsf{D}\right\rangle \in\mathbb{I}$
iff $\mathsf{G}\cap\mathsf{D}\neq\emptyset$.
\end{defn}

\begin{thm}
$\Gamma\models_{FOL}\Delta$ iff $\left\llbracket \Gamma\right\rrbracket ^{\mathcal{M}_{C}}\dttstile{}{}\left\llbracket \Delta\right\rrbracket ^{\mathcal{M}_{C}}$.\label{thm:Canonical-model-is-FOL}
\end{thm}

\begin{proof}
By Theorem~\ref{thm:FOL-is-structural-imp-space} we know that
$\Gamma\not\models_{FOL}\Delta$ iff there is a structural implication-space
model that counterexamples the inference. Hence, it suffices to show
that $\left\llbracket \Gamma\right\rrbracket ^{\mathcal{M}_{C}}\not\dttstile{}{}\left\llbracket \Delta\right\rrbracket ^{\mathcal{M}_{C}}$
iff there is a structural implication-space model, $\mathcal{M}$,
such that $\left\llbracket \Gamma\right\rrbracket ^{\mathcal{M}}\not\dttstile{}{}\left\llbracket \Delta\right\rrbracket ^{\mathcal{M}}$.

Right-to-left: It suffices to show that the canonical model is structural.
(CO) is immediate from the definition of the canonical model: for
all $\mathsf{G},\mathsf{D}\subseteq\mathbb{B}$, if $\mathsf{G}\cap\mathsf{D}\neq\emptyset$,
then $\left\langle \mathsf{G},\mathsf{D}\right\rangle \in\mathbb{I}$.
For (MO), note that, for all $\mathsf{G},\mathsf{D},\mathsf{E},\mathsf{F}\subseteq\mathbb{B}$,
if $\mathsf{G}\cap\mathsf{D}\neq\emptyset$, $(\mathsf{G}\cup\mathsf{E})\cap(\mathsf{D}\cup\mathsf{F})\neq\emptyset$.
And for (CUT), note that, for all $\mathsf{G},\mathsf{D}\subseteq\mathbb{B}$
and $\mathsf{a}\in\mathbb{B}$, if $\mathsf{G}\cap\mathsf{D}=\emptyset$,
then $(\mathsf{G}\cup\{\mathsf{a}\})\cap\mathsf{D}=\emptyset$ or
$\mathsf{G}\cap(\mathsf{D}\cup\{\mathsf{a}\})=\emptyset$. 

Left-to-right: Suppose that there is a structural implication-space
model, $\mathcal{M}$, that maps all atomic sentences to roles of
individual bearers such that $\left\llbracket \Gamma\right\rrbracket ^{\mathcal{M}}\not\dttstile{}{}\left\llbracket \Delta\right\rrbracket ^{\mathcal{M}}$.
Let $\mathcal{M}'$ be like $\mathcal{M}$ except that the objects
are natural numbers and the properties are n-tuples of natural numbers
such that $\left\llbracket a\right\rrbracket ^{\mathcal{M}'}=\left\llbracket a\right\rrbracket ^{\mathcal{M}_{C}}$
for every name in $\mathfrak{L}$ and $\left\llbracket P\right\rrbracket ^{\mathcal{M}'}=\left\llbracket P\right\rrbracket ^{\mathcal{M}_{C}}$
for every predicate in $\mathfrak{L}$. Since $\mathbb{I}$ remains
constant up to ``relettering'' of the elements, $\mathcal{M}'$
and $\mathcal{M}$ counterexample the exact same inferences. By definition,
$\left\llbracket \Gamma\right\rrbracket ^{\mathcal{M}'}\not\dttstile{}{}\left\llbracket \Delta\right\rrbracket ^{\mathcal{M}'}$
means that $\bigcup(\stackrel[i=0]{n}{\sqcup}\mathtt{g}_{i}^{+}\thinspace\sqcup\stackrel[j=0]{m}{\sqcup}\mathtt{d}_{j}^{-})\not\subseteq\mathbb{I}$,
in $\mathcal{M}'$, for $\mathtt{g}_{i}^{+}$ being the premisory
roles for the elements of $\Gamma$ and $\mathtt{d}_{j}^{-}$ being
the premisory roles for the elements of $\Delta$. Let $\left\langle \mathsf{A},\mathsf{B}\right\rangle \not\in\mathbb{I}$
be a witness of $\bigcup(\stackrel[i=0]{n}{\sqcup}\mathtt{g}_{i}^{+}\thinspace\sqcup\stackrel[j=0]{m}{\sqcup}\mathtt{d}_{j}^{-})\not\subseteq\mathbb{I}$
in $\mathcal{M}'$. By (CO) for structural models, we know that, in
$\mathcal{M}'$, for all $\mathsf{G},\mathsf{D}\subseteq\mathbb{B}$,
if $\mathsf{G}\cap\mathsf{D}\neq\emptyset$, then $\left\langle \mathsf{G},\mathsf{D}\right\rangle \in\mathbb{I}$.
So $\mathsf{A}\cap\mathsf{B}=\emptyset$. Since $\left\llbracket p\right\rrbracket ^{\mathcal{M}'}=\left\llbracket p\right\rrbracket ^{\mathcal{M}_{C}}$
for all the bearers in the two models, $\mathsf{A}\cap\mathsf{B}=\emptyset$
also holds in $\mathcal{M}_{C}$. Therefore, $\left\langle \mathsf{A},\mathsf{B}\right\rangle \not\in\mathbb{I}$
in $\mathcal{M}_{C}$. Moreover, since the interpretations of complex
sentences are composed in the same way out of the same interpretations
of atomic sentences in both models, we know that $\left\langle \mathsf{A},\mathsf{B}\right\rangle \in\bigcup(\stackrel[i=0]{n}{\sqcup}\mathtt{g}_{i}^{+}\thinspace\sqcup\stackrel[j=0]{m}{\sqcup}\mathtt{d}_{j}^{-})$
in $\mathcal{M}_{C}$. So, $\bigcup(\stackrel[i=0]{n}{\sqcup}\mathtt{g}_{i}^{+}\thinspace\sqcup\stackrel[j=0]{m}{\sqcup}\mathtt{d}_{j}^{-})\not\subseteq\mathbb{I}$
in $\mathcal{M}_{C}$. Therefore, $\left\llbracket \Gamma\right\rrbracket ^{\mathcal{M}_{C}}\not\dttstile{}{}\left\llbracket \Delta\right\rrbracket ^{\mathcal{M}_{C}}$.
\end{proof}

This means that if all the models that we use to define consequence
have implication relations that are supersets of the implication relation
of the canonical FOL model, then our consequence relation is supraclassical.
Note that we can add arbitrary further implications to the implication
relation of the canonical FOL model. These further implications can
be chosen in the ways specified above to ensure that our theory obeys
the constraints MOF, DDT, DS, LC and doesn't underwrite the rules
$\forall R$ and $\forall L$.

\begin{thm}
There is a set of implications-space models, $m$, such that $\sttstile{}{m}$
meets all of the constraints: MOF, SCL, DDT, DS, LC and it provides
counterexamples to $\forall R$ and $\forall L$.
\end{thm}

\begin{proof}
Consider the model that is like $\mathcal{M}_{C}$ but in which we
add to the implication relation of the implication-frame the pair
$\left\langle \{\mathsf{F}\mathsf{a}\},\{\mathsf{Gb}\}\right\rangle \in\mathbb{I}$,
for some distinct objects $\mathsf{a,b,c,d}$ and distinct properties
$\mathsf{F,G}$. To be explicit, we add nothing more; in particular,
we add neither $\left\langle \{\mathsf{Fa},\mathsf{Fc}\},\{\mathsf{Gb}\}\right\rangle $
nor $\left\langle \{\mathsf{Fa}\},\{\mathsf{Gd}\}\right\rangle $.
Let $\left\llbracket Fa\right\rrbracket =\mathtt{Fa}=\mathcal{R}(\mathsf{Fa})$,
and analogously for $\mathsf{Fc},\mathsf{Gb},\mathsf{Gd}$. Let $m$
be the singleton of that model.

By Proposition~\ref{prop:Frames-interpret-themselves}, $\mathtt{Fa}\dttstile{}{}\mathtt{Gb}$
but neither $\mathtt{Fa},\mathtt{Fc}\not\dttstile{}{}\mathtt{Gb}$
nor $\mathtt{Fa}\not\dttstile{}{}\mathtt{Gd}$. Hence, $Fa\sttstile{}{m}Gb$
but neither $Fa,Fc\sttstile{}{m}Gb$ nor $Fa\sttstile{}{m}Gd$. Hence,
implications in $m$ are nonmonotonic and, thus, MOF holds. Since
we merely increased the extension of $\mathbb{I}$, if our model is
a counterexample to an inference, then the canonical FOL implication-space
model is also a counterexample to that inference. Hence, given Theorem~\ref{thm:Canonical-model-is-FOL},
if $\Gamma\models_{FOL}\Delta$, then $\Gamma\sttstile{}{m}\Delta$.
So, the implications in $m$ are supraclassical and, thus, SCL holds.
By Proposition~\ref{prop:DDT-DS-LC-hold}, DDT, DS, and LC hold in
all implication-space models.

As in Proposition~\ref{prop:quantifier-rules-fail} above, it follows
from $Fa,Fc\not\sttstile{}{m}Gb$ that $\forall xFx\not\sttstile{}{m}Gb$.
Hence, we have a counterexample to $\forall L$. Similarly, from $Fa\not\sttstile{}{m}Gd$
it follows that $Fa\not\sttstile{}{m}\forall xGx$. So, we have a
counterexample to $\forall R$.
\end{proof}

We have now seen that the implication-space semantics formulated above
can meet all of our desiderata. We thus have an inferentialist semantic
theory that can validate all inferences that are valid in classical
first-order logic but also include material, defeasible inferences.
That was the goal of the paper.

\section{Conclusion}

In this paper, I have shown how to extend implication-space semantics
to include a first-order universal quantifier. In this framework,
we can model consequence relations that are nonmonotonic (MOF), supraclassical
(SCL), obey the deduction-detachment theorem (DDT), disjunction simplification
(DS), and in which conjunction behaves in a multiplicative way on
the left (LC). Since such a theory obeys MOF, SCL, and DDT, multiplicative
cut must fail in it. And since it obeys MOF, DS, and LC, it must be
hyper-intensional. The version of implication-space semantics that
I have formulated has all these properties (if one chooses one's set
of models appropriately).

To the best of my knowledge, there is no extant theory that has this
combination of properties; and these are the properties that advocates
of implication-space semantics require on philosophical grounds. I
have not considered the question whether advocates of implication-space
semantics are correct in accepting the philosophical views that call
for such a combination of properties in a semantic theory. Rather,
I have provided a a semantic theory that has these properties, and
I hope that it can help us to understand and evaluate the philosophical
motivations behind it more clearly.

\bibliographystyle{apalike}
\bibliography{Hlobil-QISS-bibliography}

\end{document}